\documentclass{book}

\newenvironment{abstract}{}{}
\usepackage{abstract}

\usepackage{arxiv}

\usepackage{hyperref}       
\usepackage{url}            
\usepackage{booktabs}       
\usepackage{amsfonts}       
\usepackage{nicefrac}       
\usepackage{microtype}      
\usepackage{amsmath}
\usepackage{amssymb}
\usepackage{amsthm}
\usepackage{mathtools}
\usepackage{comment}
\usepackage{tikz}
\usepackage{centernot}
\usepackage{verbatim}
\usepackage{blkarray}
\usepackage{graphicx}
\usepackage{easybmat}
\usepackage{multirow,bigdelim}
\usepackage[title]{appendix}
\usepackage{subcaption}
\usepackage{calrsfs}

\DeclareMathAlphabet{\pazocal}{OMS}{zplm}{m}{n}

\newcommand{\bigslant}[2]{{\raisebox{.2em}{$#1$}\left/\raisebox{-.2em}{$#2$}\right.}}

\usetikzlibrary{patterns}

\theoremstyle{plain}
\newtheorem{theorem}{Theorem}[section]
\newtheorem{lemma}[theorem]{Lemma}
\newtheorem{proposition}[theorem]{Proposition}
\newtheorem{corollary}[theorem]{Corollary}

\theoremstyle{definition}
\newtheorem{definition}[theorem]{Definition}

\theoremstyle{remark}

\title{Universal Limit Theorem for Spectra
\\of\\
iterated Inclusion-uniform Subdivisions}

\author{
  Julian M\"arte \\
  Department of Mathematics and Computer Science\\
  Philipps University Marburg\\
  35037 Marburg, Germany\\
  \texttt{mail@maertej.com}
}

\begin{document}
\maketitle

\begin{abstract}
    The main object of this work is the top-dimensional Laplacian operator of a simplicial complex $K$.
    We study its spectral limiting behavior under a given non-trivial subdivision procedure $\text{div}$. It will be shown that in case $\text{div}$ satisfies a property we call inclusion-uniformity its spectrum converges to a universal limiting distribution only depending on the dimension of $K$. This class of subdivisions contains important special cases such as the edgewise subdivision $\text{esd}_r$ for $r\geq 2$ and dimension $d = 2$ or the barycentric subdivision $\text{sd}$.
    This parallels a result of Brenti and Welker showing that the roots of $f$-polynomials of iterated barycentric subdivisions converge to a universal set of roots only depending on the dimension of $K$.
    
    Furthermore we determine the family of universal limiting functions for the particular subdivision where the top dimensional faces are replaced by a cone over their boundary. We will show that this choice of $\text{div}$ is the natural generalization of graph subdivision in the spectral sense. These limits are obtained by explicit spectral decimation of the sequence of its dual graphs which is represented as a sequence of Schreier graphs on a rooted regular tree.
    
    Finally we will point out that a generic sequence of iterated subdivisions can be realized by a sequence of graphs as in spectral analysis on fractals. We will give a construction of a self-similar sequence of graphs which dualizes the iterated application of subdivision.
\end{abstract}

\keywords{Universal Limit Theorem \and Laplacian Spectrum \and Combinatorial Laplacian \and Spectral Analysis on Fractals \and Self-similar graph sequence}

\section{Introduction}
The spectra of $k$-Laplacians of $d$-dimensional simplicial complexes, $k \leq d$, encode a variety of combinatorial and topological properties of the respective complex; cf. \cite{horak} for an overview of Laplacian operators on simplicial complexes. The case of interest for us is when $k = d$, i.e. the top-dimensional Laplacian of a $d$-dimensional simplicial complex $K$ which is defined as
$$\mathcal{L}(K) := \mathcal{L}_d(K) := \partial_d^t \partial_d$$
for the simplicial boundary operator $\partial_d$ in dimension $d$.
We are interested in how the spectrum of $\mathcal{L}(K)$ behaves (in the limit) under iterated subdivisions of $K$. We restrict ourselves to a certain intuitive subclass of geometric subdivisions in the sense of Stanley, \cite{stanley_subdiv}, which are additionally required to subdivide each face in the same way and independent of orientation. We will call them \textit{inclusion-uniform}.
The explicit definition of this class will be given in Section \ref{sec:prel}. A lot of prominent examples of geometric subdivisions are inclusion-uniform; including the edgewise subdivision of a $2$-dimensional complex and barycentric subdivision in arbitrary dimension.
For an overview of current research on subdivisions and their algebraic aspects we refer the reader to \cite{athan} and the references therein.

We consider the spectrum of a positive-semidefinite self-adjoint operator $L: \mathbb{R}^N \rightarrow \mathbb{R}^N$ as the non-decreasing right-continuous bounded (and thus $L^1$) stair-case function on $[0,1]$ given as
$$\Lambda(L) := \sum_{j=1}^N \lambda_j(L)\boldsymbol{1}_{[(j-1)/N, j/N)},$$
where $0 \leq \lambda_1(L) \leq ... \leq \lambda_N(L)$ are the ordered eigenvalues of $L$ listed with multiplicities and $\boldsymbol{1}_A$ denotes the indicator function of the set $A$. This function can be considered as a shift of the quantile function of the \textit{normalized eigenvalue counting function}\footnote{This is sometimes called the integrated density of states or the spectral CDF.}.

\begin{theorem}[Universal Limit Theorem for inclusion-uniform subdivisions]\label{th:universality}
Let $d\geq 1$ be an integer and let $\text{\normalfont div}$ be a inclusion-uniform subdivision acting non-trivially on $d$-dimensional complexes.

Then there exists a function $\Lambda_d^{\text{\normalfont (div)}} \in L^1([0,1])$ such that for every $d$-dimensional complex $K$ it holds
$$\Lambda(\mathcal{L}(\text{\normalfont div}^n K)) \xrightarrow{n\rightarrow\infty}_{L^1} \Lambda_d^{\text{\normalfont (div)}}.$$
\end{theorem}

We can associate to a complex $K$ a multitude of Laplacians. For $i\in\mathbb{N}$ we might define the $i$-up Laplacian
$\mathcal{L}_i^{\text{up}}(K) := \partial_i\partial_i^t$
and the $i$-down Laplacian
$\mathcal{L}_i^{\text{down}}(K) := \partial_i^t\partial_i$.
The $i$-dimensional Laplacian then is the sum of the $i$-up and $i$-down Laplacians
$$\mathcal{L}(K) := \mathcal{L}_i^{\text{up}}(K) + \mathcal{L}_i^{\text{down}}(K).$$
Note that in case $i = 0$ or $i = \text{dim}K$ only one of the operators is non-zero.

The existence of such universal limiting functions for $0$-up Laplacians of simplicial complexes has been studied in \cite{knill} for the particular case of $\text{div}$ being barycentric subdivision. Our main object is the $d$-down Laplacian which in general - i.e. for $d > 1$ - has no spectral relationship to the $0$-up Laplacian. However the $i$-up Laplacian has a strong spectral correlation with the $(i+1)$-down Laplacian where their spectra are identical including multiplicities except for the eigenvalue $\lambda = 0$. Thus in determining the spectral distribution of top-dimensional Laplacians we have a degree of freedom of whether to choose the $(d-1)$-up or $d$-down Laplacian to perform spectral analysis on; the choice of the $d$-down Laplacian will, however, prove to be more suitable as we won't have to compensate for changes in matrix size introduced by gluing (see Section \ref{sec:universality} for more details).

The sole dependence on $\dim K$ complements a result by Brenti and Welker, \cite{welker}, showing that the roots of $f$-polynomials of the sequence of iterated barycentric subdivisions of a complex converge to a universal set of roots only depending on the dimension of $K$. Effects of this kind can be attributed to the dominance of local features introduced by the repeated subdivision.

Having established the existence of a universal limiting function a natural question to ask is whether we can determine this function for given $d$ and a inclusion-uniform subdivision $\text{div}$. This question can be reduced to one on (signed) graph spectra when considering the $d$-Laplacian as the graph Laplacian of the $d$-dual graph of $K$ (as a signed graph). The subdivision operation then induces an operation on the dual graphs by replacing every vertex by a copy of a "fundamental graph" and joining them appropriately by edges. These joining operations in turn depend on the edges of the given graph. We thus seek to analyze the effect a graph operation induced by subdivision has on the spectrum.
A variety of such spectral effects of common graph operations is summarized in \cite{spectral_effects, cvetkovic_subdiv_graph}, with one particular example of a unary graph operation being the (barycentric) subdivision of a graph (regarded as a $1$-dimensional simplicial complex).

We say that a graph operation $S: G\mapsto S(G)$ admits \textit{"spectral decimation"} if there is a rational function $f_S$ such that the spectrum of $S(G)$ consists of the solutions $\mu\in\mathbb{R}$ of the equations
$$\lambda = f(\mu)$$
for $\lambda$ in the spectrum of $G$ (with eventual adjustment of multiplicities and up to some "small" exceptional set $\mathcal{E}$). Thus $S(G)$ only carries spectral information stemming from either $G$ or $S$ (up to $\mathcal{E}$). The notion of spectral decimation originates from fractal analysis, e.g. \cite{jordan}. In Section \ref{sec:fractal} we will describe how iterated subdivisions fit the framework of spectra of self-similar graph sequences.
Graph subdivision is one case for which a spectral decimation holds as long as the input graph is regular, \cite{spectral_effects}.

In order for spectral decimation to be applicable iteratively we need to assume the initial graph $G$ to be $2$-regular. Then $S(G)$ will again be $2$-regular. For $r$-regular graphs $G$, $r \geq 3$ $S(G)$ is not regular anymore. However as the limiting distribution does not depend on $G$ (as we will see in Theorem \ref{th:universality}) we can pick the initial setting $G$ at will - in particular we might choose it to be $2$-regular.

Note that the Laplacian of a $2$-regular graph can be written as
$$L(G) = 2\cdot I - A(G)$$
where $A(G)$ denotes the adjacency matrix of $G$ and $I$ is the identity matrix of proper size. Since $2$-regularity is preserved under subdivision so is the relation between Laplacian and adjacency matrix. As a consequence of this the sequences of spectra of Laplacians and adjacency matrices are related over an affine-linear transformation.
In this particular case we obtain that the eigenvalues of the adjacency matrix of $S(G)$ are given by the roots of the polynomial equation
$$f_A(\zeta) = \zeta^2 - 2 = \lambda$$
for $\lambda$ running over the set of eigenvalues of the adjacency matrix associated to the initial graph $G$, as shown in \cite{spectral_effects} for example. In case such a decimation holds we call $f_A$ the spectral decimation map.
Analogously the spectral decimation map for the Laplacian spectrum in the $2$-regular graph case is given by
$$f_L(\zeta) = \zeta(4 - \zeta)$$
which can be seen through substituting by the affine-linear transformation of spectra discussed above.

There are many subdivision procedures $\text{div}$ which coincide with $S$ on $1$-dimensional simplicial complexes. One natural question to ask is which of those generalizes $S$ in a spectral sense. In Section \ref{sec:cone} we will find a higher-dimensional analogue of the above decimation for the subdivision operation $\text{cd}$ shown in Figure \ref{fig:cd}. For a complex $K$ of dimension $d$ $\text{cd}K$ is obtained from the $(d-1)$-skeleton $K^{(d-1)}$ by adding the barycenter $v_\sigma$ of every facet $\sigma\in F_d(K)$ together with the faces $v_\sigma\cup\tau$ for $\tau < \sigma$.
As we will see from this concrete example the determination of an exact spectral decimation is much more involved in this case.

This work is structured as follows: Section \ref{sec:prel} gives an introduction to the main objects and frameworks used in the course of this paper. The Universal Limit Theorem, Theorem \ref{th:universality}, is proven in Section \ref{sec:universality}. The universal limit of the subdivision $\text{cd}$ is determined by Theorem \ref{th:conespectrum} in Section \ref{sec:cone}. Lastly in Section \ref{sec:fractal} we point out the strong relation the spectral theory for iterated subdivision has to fractal theory by giving a construction procedure of fractals dualizing subdivision of a complex. 

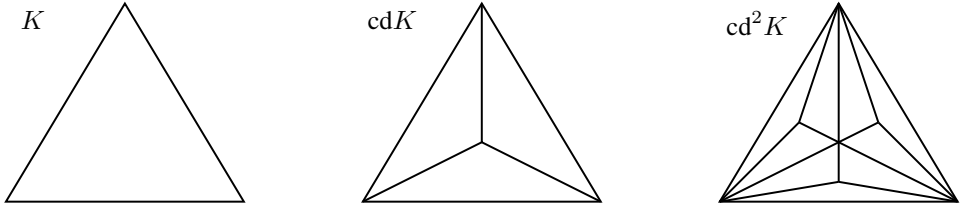
\begin{figure}
    \centering
    \tikzset{every picture/.style={line width=0.75pt}} 

\begin{tikzpicture}[x=0.75pt,y=0.75pt,yscale=-1,xscale=1]

\draw    (150,70) -- (210,170) -- (90,170) -- cycle ;
\draw    (330,70) -- (390,170) -- (270,170) -- cycle ;
\draw    (510,70) -- (570,170) -- (450,170) -- cycle ;
\draw    (330,70) -- (330,140) -- (390,170) ;
\draw    (330,140) -- (270,170) ;
\draw    (510,70) -- (510,140) -- (570,170) ;
\draw    (510,140) -- (450,170) ;
\draw    (450,170) -- (510,160) ;
\draw    (510,140) -- (510,160) ;
\draw    (510,160) -- (570,170) ;
\draw    (490,130) -- (450,170) ;
\draw    (510,140) -- (490,130) ;
\draw    (510,70) -- (490,130) ;
\draw    (530,130) -- (570,170) ;
\draw    (510,70) -- (530,130) ;
\draw    (510,140) -- (530,130) ;

\draw (96,72.4) node [anchor=north west][inner sep=0.75pt]    {$K$};
\draw (271,72.4) node [anchor=north west][inner sep=0.75pt]    {$\text{cd} K$};
\draw (452,72.4) node [anchor=north west][inner sep=0.75pt]    {$\text{cd}^{2} K$};

\end{tikzpicture}
    \caption{The first $3$ complexes of the sequence of iterated application of $\text{cd}$ for $d = 2$ for the initial complex $K = \Delta^{(2)}$ - the standard-$2$-simplex.}
    \label{fig:cd}
\end{figure}

\section{Preliminaries}\label{sec:prel}
\subsection*{Basics on simplicial complexes}
The following objects are defined in \cite{horak} (even though the notation might vary). A thorough introduction to simplicial topology and geometry can be found in \cite{munkres}.

A simplicial complex $K$ on a finite vertex set $V$ is a collection of subsets of $V$ downwards-closed under $\subset$, i.e. if $A\subset B \in K$ then also $A\in K$. We denote by $F_i(K)$ the collection of sets of $K$ of size $i+1$ and call those elements \textit{$i$-dimensional faces} of $K$. The dimension of $K$ is the maximum dimension of a face in $K$.

We call a simplicial complex $K$ oriented if for every face $\tau\in K$ we fix a linear ordering of the vertices of $\tau$. Two orientations of $K$ are said to be equivalent if for every $\tau\in K$ the orderings fixed for the vertices of $\tau$ are obtained from each other by applying an even permutation, thus partitioning orientations of $\tau$ in two equivalence classes. If the orientation fixed for $\tau$ is relevant we emphasize this by writing $[\tau]$ instead of $\tau$.
The orientation opposite to $[\tau]$ is denoted by $-[\tau]$.
We denote by $C_i(K)$ the $\mathbb{R}$-vector space over the basis elements $\{e_{[\tau]} ~|~ \tau\in F_i(K)\}$ and call $C_i(K)$ the chain groups of $K$ with coefficients in $\mathbb{R}$. The opposite orientations of elements of $F_i(K)$ are interpreted as elements of $C_i(K)$ by
$$e_{-[\tau]} = -e_{[\tau]}.$$

$C_\bullet(K)$ becomes a chain complex with the usual simplicial boundary operator,
$$\partial_i[v_0,...,v_i] := \sum_{j=0}^i (-1)^j e_{[v_0,...,\hat{v_j},...,v_i]}.$$
Further we equip $C_i(K)$ with the standard inner product and denote by $\partial_i^\ast$ the operator adjoint to $\partial_i$ with respect to the chosen inner products.

Now we are ready to define the Laplacian operators in different dimensions.
\begin{definition}
Let $K$ be an oriented simplicial complex and $i\in\mathbb{N}$ then we define
\begin{itemize}
    \item the $i$-up Laplacian to be
    $$\mathcal{L}_i^{\text{up}}(K) := \partial_{i+1}\partial_{i+1}^*,$$
    \item the $i$-down Laplacian to be
    $$\mathcal{L}_i^{\text{down}}(K) := \partial_i^*\partial_i$$
    and
    \item the $i$-Laplacian to be
    $$\mathcal{L}_i(K) := \mathcal{L}_i^{\text{down}}(K) + \mathcal{L}_i^{\text{up}}(K).$$
\end{itemize}
\end{definition}

Note that by definition for a $d$-dimensional complex $K$ it holds $\mathcal{L}_d^{\text{up}}(K) = 0$ and thus
$$\mathcal{L}_d(K) = \mathcal{L}_d^{\text{down}}(K).$$
We will describe to combinatorics decoded by $\mathcal{L}_d^{\text{down}}(K)$ in the following. 

In order to model higher-dimensional adjacencies in $K$ we will say $\tau, \tau'\in F_{i+1}(K)$ are \textit{$(i+1)$-down neighbors} if they share a common $i$-face, i.e. $\tau\cap\tau'\in F_i(K)$.
The $i$-dual graph $\Gamma^{(i)}(K)$ of a complex $K$ for us then is the graph on vertex set $F_i(K)$ with edge set $E$ modelling the $i$-down adjacency, i.e. $\{\tau, \tau'\}\in E$ iff $\tau\cap\tau'\in F_{i-1}(K)$.

A signed graph $G = (V,E,\sigma)$ is an undirected graph $G$ with a function $\sigma: E \rightarrow \{\pm 1\}$ signing each edge. The degree of a vertex in a signed graph is the degree of a vertex in the underlying undirected graph $G = (V,E)$. Order the vertices of $G$ arbitrarily and denote by $D(G)$ the diagonal matrix of degrees of vertices of $G$, $D(G)_{ii} = \deg (v_i)$, and $A(G)$ the signed adjacency matrix of $G$,
$$A(G)_{ij} = \begin{cases}
0 &, \{i,j\}\notin E\\
\sigma(\{i,j\}) &, \{i,j\}\in E
\end{cases}.$$
Note that the Laplacian of a simplicial complex then is a natural generalization of the graph Laplacian
$$\mathcal{L}(G) := D(G) + A(G)\footnote{We obtain the common Laplacian operator for the sign $\sigma \equiv -1$ in this definition.}$$
in the following sense:

By Proposition 3.3.3 of \cite{goldberg2002combinatorial} we have that for $K$ an oriented simplicial complex it holds that
$$\mathcal{L}_i^{\text{down}}(K) = \mathcal{L}(\Gamma^{(i)}(K), \sigma)$$
where the sign map $\sigma: E \rightarrow \{\pm 1\}$ is given by
$$\sigma(\{\tau, \tau'\}) := \delta_\tau(\tau\cap\tau')\cdot \delta_{\tau'}(\tau\cap\tau')$$
for $\delta_\tau: ~ F_{i-1}(\tau) \rightarrow \{\pm 1\}$ given as
$$\delta_\tau(\nu) := \langle \partial_i e_{[\tau]}, e_{[\nu]}\rangle,$$
i.e. the coefficient of $e_{[\nu]}$ in $\partial_i e_{[\tau]}$. This definition measures if the induced orientation of $[\tau]$ over $\partial_i$ coincides with the orientation $[\nu]$ fixed for $\nu$. Thus if the induced orientations of $[\tau]$ and $[\tau']$ on $\tau\cap\tau'$ are the same we obtain
$$A(\Gamma^{(i)}(K), \sigma)_{\tau,\tau'} = 1$$
and if they differ
$$A(\Gamma^{(i)}(K), \sigma)_{\tau, \tau'} = -1.$$
In case $\tau$ and $\tau'$ are not even $i$-down neighbors the adjacency operator is zero in this entry.

A special case where this point of view is particularly interesting is the case of \textit{orientable} complexes. We say a pure $d$-dimensional simplicial complex $K$ is orientable if there is an orientation of $K$ such that every pair of $d$-down neighboring faces $\{\tau, \tau'\}$ induces opposing orientations on $\tau\cap\tau'$, i.e. in the above notation
$$\sigma(\{\tau, \tau'\}) = -1.$$
Thus $\sigma \equiv -1$ and the $d$-dual graph is just an undirected graph with $\mathcal{L}_d(K)$ being its ordinary graph Laplacian.
As mentioned above in what follows we will consider the case $i = d = \dim K$ and will denote the top-dimensional Laplacian by $\mathcal{L}(K) := \mathcal{L}_d(K)$.

\subsection*{Asymptotic spectral analysis}
\begin{definition}
Let $L$ be a Hermitian $N\times N$ matrix. We call the $L^1$-function
$$\Lambda(L) = \sum_{j=1}^{N-1} \lambda_j(L)\boldsymbol{1}_{[\frac{j-1}{N}, \frac{j}{N})}$$
the \textit{shifted spectral quantile function of $L$}.
\end{definition}

Note that this notion originates from the fact that $\Lambda(L)$ is a shift of the quantile function of the spectral CDF
$$F_L(x) = \frac{1}{N}\#\{i\in [N] ~|~ \lambda_i(L) \leq x\}.$$
The quantile function of $F_L$ is given as
$$Q_L(p) = \sum_{j=1}^{N-1} \lambda_j(L) \boldsymbol{1}_{[j/N, (j+1)/N)} + \lambda_N(L)\boldsymbol{1}_{\{1\}}$$
and thus $\Lambda(L)$ is the shift
$$\Lambda(L)(p) = Q_L(\min(p + 1/N,1)).$$

For the rest of this work we will denote by $||\cdot||_1^{\text{norm}}$ the normalized $L^1$-norm of matrices, i.e. for $A\in\mathbb{C}^{N\times N}$
$$||A||_1^{\text{norm}} := \frac{||A||_1}{N},$$
for the common $L^1$ matrix-norm.

The following proposition is \cite[inequality (1.2)]{wielandt}; we refer the reader to the sources mentioned in the introduction therein.
\begin{proposition}[$1$-Wielandt-Hoffman inequality, 
\cite{wielandt}]
Let $L, E\in M_N(\mathbb{C})$ be Hermitian matrices. It holds that
$$\sum_{j=1}^N |\lambda_j(L + E) - \lambda_j(L)| \leq \sum_{j=1}^k \sigma_j(E) = ||E||_{S^1},$$
where $\sigma_j(E)$ denotes the $j$-th singular value of $E$ and $||\cdot||_{S^1}$ is the Schatten-$1$-norm.
\end{proposition}
Together with the fact that $||\cdot||_{S^1} \leq ||\cdot||_1$\footnote{Which can easily be seen from the fact that the Schatten-$1$-norm is the nuclear norm for $2$-tensors as mentioned in \cite{friedland} and the references therein.} we obtain the following useful corollary.
\begin{corollary}\label{cor:wielandt}
Let $L, E\in M_N(\mathbb{C})$ be Hermitian matrices. It holds that
$$||\Lambda(L + E) - \Lambda(L)||_{L^1} \leq ||E||_1^{\text{\normalfont norm}},$$
where $||\cdot||_{L^1}$ denotes the $L^1([0,1])$-norm.
\end{corollary}

We will use this inequality in the proof of Theorem \ref{th:universality} in a similar manner to how related statements are used for the use of \textit{approximating class of sequences} in GLT matrix theory, cf. \cite{glt}.

\subsubsection{Tools for explicit spectral analysis}\label{sssec:spectral}
In order to exactly compute certain determinants or inverses under low-rank perturbations in Section \ref{sec:cone} we will use the following two convenient results.

\begin{lemma}[Sherman-Morrison-Woodbury formula, \cite{sherman_morrison, woodbury}]\label{lemm:SMW}
Let $A\in \mathbb{R}^{n\times n}$, $U\in \mathbb{R}^{n\times m}$, $V\in\mathbb{R}^{m\times n}$. Assume $A$ and $I_m - VA^{-1}U$ are invertible. Then the inverse of $A - UV$ is given as
$$(A-UV)^{-1} = A^{-1} + A^{-1} U(I_m - VA^{-1}U)^{-1}VA^{-1}.$$
In particular for $m = 1$, $U = u\in\mathbb{R}^n, V = v\in \mathbb{R}^n$ we obtain the original Sherman-Morrison formula
$$(A-uv^t)^{-1} = A^{-1} + \frac{A^{-1}uv^t A^{-1}}{1 - v^t A^{-1} u}.$$
\end{lemma}

\begin{lemma}[Matrix Determinant Lemma, Theorem 18.1.1 of \cite{MDL}]\label{lemm:MDL}
Let $A\in\mathbb{R}^{n\times n}$, $B\in \mathbb{R}^{m\times m}$, $U\in\mathbb{R}^{n\times m}$, $V\in\mathbb{R}^{m\times n}$. It holds that
$$\det(A + UBV) = \det A \det B \det (B^{-1} + VA^{-1}U).$$
In the particular case of $m = 1$, $B = 1$ and vectors $U = u\in\mathbb{R}^m$, $V = v\in\mathbb{R}^m$ we obtain
$$\det(A + uv^t) = (1 + v^t A^{-1} u)\det A.$$
\end{lemma}

The following result will help us resolve block matrix determinants.

\begin{lemma}[Schur-Renormalization, Theorem 13.3.8. of \cite{MDL}]\label{lemm:SR}
Let $A\in \mathbb{R}^{n\times n}$, $B\in \mathbb{R}^{n\times m}$, $C\in\mathbb{R}^{m\times n}$, $D\in\mathbb{R}^{m\times m}$. Then it holds that
$$\det \begin{pmatrix}
A & B\\
C & D
\end{pmatrix} = \det \begin{pmatrix}
D & C\\
B & A
\end{pmatrix} = \det A \det (D - CA^{-1}B).$$
\end{lemma}

\subsection*{Iterated subdivisions of simplicial complexes}
We will be using the notion of geometric subdivisions, cf. \cite{stanley_subdiv}, \cite[p. 83]{munkres}. To this end we assume every simplicial complex to be a geometric simplicial complex, i.e. be embedded in some euclidean space for the rest of this subsection. This is no obstruction on the simplicial complex as every abstract simplicial complex has a \textit{geometric realization}, cf. \cite[Theorem 3.1]{munkres}. We will thus use the notions of geometric and abstract complexes interchangably - assuming to have fixed some geometric realization of the initial complexes. We assume the standard-$d$-simplex to be realized as $\text{conv}(e_1,...,e_{d+1})\subset \mathbb{R}^{d+1}$ for the standard basis $\{e_1,...,e_{d+1}\}$.

Furthermore let $d$ be a fixed dimension.

\begin{definition}
A procedure $\text{div}$ associating to a $d$-dimensional geometric complex $K$ a geometric complex $\text{div}K$ is called a \textit{subdivision procedure} if the following conditions hold:
\begin{itemize}
    \item[(i)] Every simplex of $\text{div}K$ is contained in some simplex of $K$.
    \item[(ii)] Every simplex of $K$ is the union of finitely many simplices of $\text{div}K$.
\end{itemize}
\end{definition}
It is well-known that every subdivision $\text{div}K$ induces a map $s: \text{div} K \rightarrow K$ associating to a face $\sigma\in \text{div}K$ the smallest face $\tau\in K$ such that $\sigma$ is contained in $\tau$.
The subcomplexes $\text{div}_K(\tau) := s^{-1}(2^{\tau}) \leq \text{div} K$ are called \textit{restrictions of $\text{div} K$ to $\tau$} for $\tau\in F_i(K)$. $\text{div}_K\tau$ corresponds to the subdivision of $\tau$ as a face in $K$.

\begin{definition}\label{def:barycentric}
A subdivision procedure $\text{div}$ is said to be \textit{inclusion-uniform} if for every $d$-dimensional complex $K$ and face $\tau\in K$ of dimension $i$, $i\in \{0,...,d\}$, every possible identification of $\tau$ with $\Delta_i$ extends to an isomorphism between $\text{div}_K\tau$ and $\text{div}\Delta_i$, i.e. let $\tau = \text{conv}(v_0,...,v_i)$ and given a bijection $f: \{v_0,...,v_i\} \rightarrow \{e_1,...,e_{i+1}\} = F_0(\Delta_i)$ there exists a unique simplicial isomorphism $\tilde{f}: \text{div}_K\tau \rightarrow \text{div}\Delta_i$ such that $\tilde{f}_{\vert_{\{v_0,...,v_i\}}} = f$.
\end{definition}

An immediate consequence of the definition is that for two complexes $K$ and $L$ and dedicated faces $\tau\in F_i(K)$, $\sigma\in F_i(L)$ with a bijective vertex map $\pi: F_0(\tau) \rightarrow F_0(\sigma)$ there is a unique simplicial isomorphism
$$\tilde{\pi}: \text{div}_K\tau \rightarrow \text{div}_L\sigma$$
such that $\tilde{\pi}(v) = \pi(v)$ for $v\in F_0(\tau)$.

Note that the barycentric subdivision - $\text{sd}$ defined as the complex of increasing sequences of faces (so called \textit{flags}) in $K$ is itself inclusion-uniform. inclusion-uniform subdivisions are uniquely determined by a sequence of subdivisions $\text{div} \Delta_i$ of $\Delta_i$, $i\in\mathbb{N}$, such that the restriction of $\text{div}\Delta_i$ to $\sigma$ is isomorphic to $\text{div}\Delta_{i-1}$ for every $\sigma\in F_{i-1}(\Delta_i)$. Such a sequence is called a \textit{subdivision scheme} in the following. As the face number of the subdivided $i$-simplex is intrinsic to $\text{div}$ in what follows we will write
$$f_i(\text{div}) := f_i(\text{div}\Delta_i),$$
i.e. $f_i(\text{div})$ counts the number of facets the standard $i$-simplex gets subdivided in.

In particular inclusion-uniform subdivisions are a special case of repeatable subdivisions, i.e. subdivisions which can be applied arbitrarily often to any initial complex $K$. This can be seen by describing the procedure of subdividing according to $\text{div}$ in an iterative manner. Let $K$ be a given $d$-dimensional complex, then the isomorphism type of $\text{div}K$ can be obtained from $K$ and a subdivision scheme $\{\text{div}\Delta_i\}_{i=0,...,d}$ by the following inductive construction: Set $K_0 = F_0(K)$.\\
Now let $K_i$ be constructed for some $0\leq i < d$. For every $\tau\in F_{i+1}(K)$ let $\tau = \text{conv}(v_0,...,v_{i+1})$. Identify $\{v_0,...,v_{i+1}\}$ with $\{e_1,...,e_{i+2}\}$ arbitarily and let $\tilde{f}$ denote the isomorphism of $\text{div}_K\tau$ and $\text{div}\Delta_{i+1}$ induced by this identification. Add to $K_i$ the pre-image of $\tilde{f}$ and proceed with the next $(i+1)$-face of $K$. This way we obtain $K_{i+1}$.

Note that since $\text{div}$ is inclusion-uniform the construction does not depend on the chosen identifications and thus $K_d$ is isomorphic to $\text{div}K$. It is apparent by this procedure that $\text{div}$ is repeatable.

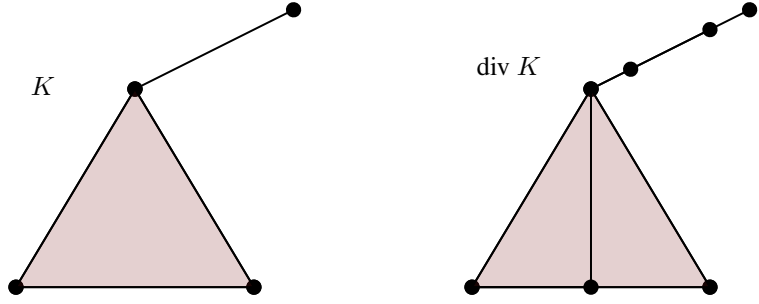
\begin{figure}
    \centering
    \tikzset{every picture/.style={line width=0.75pt}} 

\begin{tikzpicture}[x=0.75pt,y=0.75pt,yscale=-1,xscale=1]

\draw    (180,70) -- (260,30) ;
\draw [shift={(260,30)}, rotate = 333.43] [color={rgb, 255:red, 0; green, 0; blue, 0 }  ][fill={rgb, 255:red, 0; green, 0; blue, 0 }  ][line width=0.75]      (0, 0) circle [x radius= 3.35, y radius= 3.35]   ;
\draw [shift={(180,70)}, rotate = 333.43] [color={rgb, 255:red, 0; green, 0; blue, 0 }  ][fill={rgb, 255:red, 0; green, 0; blue, 0 }  ][line width=0.75]      (0, 0) circle [x radius= 3.35, y radius= 3.35]   ;
\draw    (180,70) -- (240,170) ;
\draw [shift={(240,170)}, rotate = 59.04] [color={rgb, 255:red, 0; green, 0; blue, 0 }  ][fill={rgb, 255:red, 0; green, 0; blue, 0 }  ][line width=0.75]      (0, 0) circle [x radius= 3.35, y radius= 3.35]   ;
\draw [shift={(180,70)}, rotate = 59.04] [color={rgb, 255:red, 0; green, 0; blue, 0 }  ][fill={rgb, 255:red, 0; green, 0; blue, 0 }  ][line width=0.75]      (0, 0) circle [x radius= 3.35, y radius= 3.35]   ;
\draw    (120,170) -- (240,170) ;
\draw [shift={(240,170)}, rotate = 0] [color={rgb, 255:red, 0; green, 0; blue, 0 }  ][fill={rgb, 255:red, 0; green, 0; blue, 0 }  ][line width=0.75]      (0, 0) circle [x radius= 3.35, y radius= 3.35]   ;
\draw [shift={(120,170)}, rotate = 0] [color={rgb, 255:red, 0; green, 0; blue, 0 }  ][fill={rgb, 255:red, 0; green, 0; blue, 0 }  ][line width=0.75]      (0, 0) circle [x radius= 3.35, y radius= 3.35]   ;
\draw    (180,70) -- (147.06,124.9) -- (120,170) ;
\draw [shift={(120,170)}, rotate = 120.96] [color={rgb, 255:red, 0; green, 0; blue, 0 }  ][fill={rgb, 255:red, 0; green, 0; blue, 0 }  ][line width=0.75]      (0, 0) circle [x radius= 3.35, y radius= 3.35]   ;
\draw [shift={(180,70)}, rotate = 120.96] [color={rgb, 255:red, 0; green, 0; blue, 0 }  ][fill={rgb, 255:red, 0; green, 0; blue, 0 }  ][line width=0.75]      (0, 0) circle [x radius= 3.35, y radius= 3.35]   ;
\draw [fill={rgb, 255:red, 123; green, 20; blue, 20 }  ,fill opacity=0.2 ]   (180,70) -- (240,170) -- (120,170) -- cycle ;
\draw    (410,70) -- (490,30) ;
\draw [shift={(490,30)}, rotate = 333.43] [color={rgb, 255:red, 0; green, 0; blue, 0 }  ][fill={rgb, 255:red, 0; green, 0; blue, 0 }  ][line width=0.75]      (0, 0) circle [x radius= 3.35, y radius= 3.35]   ;
\draw [shift={(410,70)}, rotate = 333.43] [color={rgb, 255:red, 0; green, 0; blue, 0 }  ][fill={rgb, 255:red, 0; green, 0; blue, 0 }  ][line width=0.75]      (0, 0) circle [x radius= 3.35, y radius= 3.35]   ;
\draw    (410,70) -- (470,170) ;
\draw [shift={(470,170)}, rotate = 59.04] [color={rgb, 255:red, 0; green, 0; blue, 0 }  ][fill={rgb, 255:red, 0; green, 0; blue, 0 }  ][line width=0.75]      (0, 0) circle [x radius= 3.35, y radius= 3.35]   ;
\draw [shift={(410,70)}, rotate = 59.04] [color={rgb, 255:red, 0; green, 0; blue, 0 }  ][fill={rgb, 255:red, 0; green, 0; blue, 0 }  ][line width=0.75]      (0, 0) circle [x radius= 3.35, y radius= 3.35]   ;
\draw    (350,170) -- (470,170) ;
\draw [shift={(470,170)}, rotate = 0] [color={rgb, 255:red, 0; green, 0; blue, 0 }  ][fill={rgb, 255:red, 0; green, 0; blue, 0 }  ][line width=0.75]      (0, 0) circle [x radius= 3.35, y radius= 3.35]   ;
\draw [shift={(350,170)}, rotate = 0] [color={rgb, 255:red, 0; green, 0; blue, 0 }  ][fill={rgb, 255:red, 0; green, 0; blue, 0 }  ][line width=0.75]      (0, 0) circle [x radius= 3.35, y radius= 3.35]   ;
\draw    (410,70) -- (377.06,124.9) -- (350,170) ;
\draw [shift={(350,170)}, rotate = 120.96] [color={rgb, 255:red, 0; green, 0; blue, 0 }  ][fill={rgb, 255:red, 0; green, 0; blue, 0 }  ][line width=0.75]      (0, 0) circle [x radius= 3.35, y radius= 3.35]   ;
\draw [shift={(410,70)}, rotate = 120.96] [color={rgb, 255:red, 0; green, 0; blue, 0 }  ][fill={rgb, 255:red, 0; green, 0; blue, 0 }  ][line width=0.75]      (0, 0) circle [x radius= 3.35, y radius= 3.35]   ;
\draw [fill={rgb, 255:red, 123; green, 20; blue, 20 }  ,fill opacity=0.2 ]   (410,70) -- (470,170) -- (350,170) -- cycle ;
\draw    (410,70) -- (410,170) ;
\draw    (350,170) -- (410,170) ;
\draw [shift={(410,170)}, rotate = 0] [color={rgb, 255:red, 0; green, 0; blue, 0 }  ][fill={rgb, 255:red, 0; green, 0; blue, 0 }  ][line width=0.75]      (0, 0) circle [x radius= 3.35, y radius= 3.35]   ;
\draw    (410,70) -- (430,60) ;
\draw [shift={(430,60)}, rotate = 333.43] [color={rgb, 255:red, 0; green, 0; blue, 0 }  ][fill={rgb, 255:red, 0; green, 0; blue, 0 }  ][line width=0.75]      (0, 0) circle [x radius= 3.35, y radius= 3.35]   ;
\draw    (430,60) -- (470,40) ;
\draw [shift={(470,40)}, rotate = 333.43] [color={rgb, 255:red, 0; green, 0; blue, 0 }  ][fill={rgb, 255:red, 0; green, 0; blue, 0 }  ][line width=0.75]      (0, 0) circle [x radius= 3.35, y radius= 3.35]   ;

\draw (126,62.4) node [anchor=north west][inner sep=0.75pt]    {$K$};
\draw (351,52.4) node [anchor=north west][inner sep=0.75pt]    {$\text{div} \ K$};

\end{tikzpicture}
    \caption{Subdivision procedure which is not inclusion-uniform. See how there are edges subdivided by one or two vertices or not even subdivided at all. Obviously those are not isomorphic as simplicial complexes. Note also that the subdivision of the $2$-face is not rotational invariant which would be necessary for $\text{div}$ to be inclusion-uniform.}
    \label{fig:nonbary}
\end{figure}

Furthermore in what follows we will call $\text{div}$ \textit{finitely ramified} or \textit{of finite ramification} if
$$f_{d-1}(\text{div}) = 1,$$
i.e. if $\text{div}$ only acts non-trivially on $d$-faces. This notion is inspired by the fractal concept underlying the spectral theory we are discussing in the upcoming section, see Section \ref{sec:fractal} for this connection.

In order to prove the main theorem of this paper we will need another operation on simplicial complexes.

\subsubsection*{Gluing and inclusion-uniform subdivisions}
We now consider two formally disjoint $d$-dimensional complexes $K$ and $L$.
Let $\mathcal{G}$ be a relation on the set $F_0(K)\times F_0(L)$. We write $v\mathcal{G}w$ for $\mathcal{G}(v,w)$.

\begin{definition}
We say that $\mathcal{G}$ defines a \textit{gluing} of $K$ and $L$ if the following holds:
\begin{itemize}
    \item For every vertex $v\in F_0(K)$ there is at most one vertex $w\in F_0(L)$ such that $v\mathcal{G}w$ and vice versa, i.e. let
    $$G_0(K) := \{v\in F_0(K) ~|~ \exists_{w\in F_0(L)}: ~ v\mathcal{G} w\}$$
    and $G_0(L)$ analogously, then there is a bijection $\varphi: G_0(K)\rightarrow G_0(L)$ such that $v\mathcal{G}w$ iff $w = \varphi(v)$.
    \item $\varphi$ induces a well-defined simplicial isomorphism between $K_{\vert_{G_0(K)}}$ and $L_{\vert_{G_0(L)}}$.
\end{itemize}
\end{definition}

In the following we denote by $G(K)$ and $G(L)$ the vertex-induced subcomplexes $K_{\vert_{G_0(K)}}$ and $L_{\vert_{G_0(L)}}$.

Note that since $\varphi$ induces a well-defined simplicial isomorphism $\tilde{\varphi}$ between $G(K)$ and $G(L)$ the \textit{glued complex}
$$K\mathcal{G}_\ast L := \bigslant{K\sqcup L}{\sim_\mathcal{G}}$$
is well-defined for $\sim_\mathcal{G}$ being the relation on $K\times L$ generated by the relations $\sigma\sim_\mathcal{G}\tilde{\varphi}(\sigma)$ for $\sigma\in G(K)$.
Denote for a gluing $\mathcal{G}$ by $r_i(\mathcal{G})$ the number of non-trivial relations
$$\tau\sim_{\mathcal{G}}\sigma$$
for $\tau\in F_i(K)$, $\sigma\in F_i(L)$.

Note that gluing procedures of more than two complexes can be defined inductively. In this case we write
$$\mathcal{G}_\ast(K_1,...,K_\ell)$$
for the glued complex.

In the following let $s_K, s_L, s$ denote the subdivision maps of $K, L$ and $K\mathcal{G}_\ast L$, respectively.
Given two complexes $K$ and $L$ let $\iota_K: ~ \text{div}K \rightarrow \text{div}(K\mathcal{G}_* L)$ and $\iota_L: ~\text{div}L \rightarrow \text{div}(K\mathcal{G}_\ast L)$ be the natural geometrical inclusions induced by the inclusions $\iota_K':K\rightarrow K\mathcal{G}_\ast L$ and $\iota_L': L \rightarrow K\mathcal{G}_\ast L$ over the isomorphism derived from Definition \ref{def:barycentric}, i.e. for every face $\tau = \{v_0,...,v_i\}\in K$ we define
$$\iota_K{}_{\vert_{\text{div}_K\tau}} := \widetilde{(\iota_K')_{\vert_{\tau}}}.$$
This definition is compatible along boundaries and thus assembles to a well-defined injective function (since $s_K^{-1}(\tau)$ are disjoint sets for distinct $\tau$'s).

Obviously two faces in $\text{div} K$ and $\text{div} L$ can only be mapped onto the same face by $\iota_K$ and $\iota_L$ in $\text{div}(K\mathcal{G}_\ast L)$ if they lie in some face in $G(K)$ or $G(L)$, respectively. Furthermore the union of images $\text{im}\iota_K \cup \text{im}\iota_L$ exhausts $\text{div}(K\mathcal{G}_\ast L)$ and so $\text{div}(K\mathcal{G}\ast L)$ can be obtained as a gluing from $\text{div}K$ and $\text{div} L$ by identifying faces which are mapped the same face in $\text{div}(K\mathcal{G}_\ast L)$.

This gluing procedure is precisely given by the relation $\mathcal{G}'$ generated by
$$v\mathcal{G}' w$$
for $v\in F_0(\text{div} K)$ and $w\in F_0(\text{div} L)$ if $\iota_K(v) = \iota_L(w)$.
Thus
$$G_0(\text{div} K) = F_0(s^{-1}(G(K))), ~~ G_0(\text{div} L) = F_0(s^{-1}(G(L)))$$
and the bijection $\varphi': G_0(\text{div}K) \rightarrow G_0(\text{div} L)$ satisfying the two conditions of a gluing is given by
\begin{equation}\label{eq:gluing_bij}
\varphi'(v) := \widetilde{(\iota_L')_{\vert_\sigma}}^{-1}\circ \widetilde{(\iota_K')_{\vert_\tau}}(v)
\end{equation}
for $\tau := s_K^{-1}(v)$ and $\sigma := \iota_L'^{-1}\circ\iota_K'(\tau)$.  By definition the simplicial map defined by $\varphi'$ is compatible along boundaries and yields an isomorphism of the respective vertex-induced subcomplexes.

By all the above we have
$$\text{div}(K\mathcal{G}_\ast L) \cong (\text{div} K)\mathcal{G}'_\ast (\text{div} L).$$

Note that assuming $r_d(\mathcal{G}) = 0$, i.e. $\mathcal{G}$ does not identify facets of $K$ and $L$ with each other, the newly defined gluing $\mathcal{G}'$ satisfies
$$r_{d-1}(\mathcal{G}') = f_{d-1}(\text{div})\cdot r_{d-1}(\mathcal{G}).$$

We summarize this procedure in the following proposition for later use.
\begin{proposition}[Subdivision gluing]\label{prop:gluing}
Let $\text{\normalfont div}$ denote a inclusion-uniform subdivision. Given a gluing $\mathcal{G}$ of $K$ and $L$ satisfying $r_d(\mathcal{G}) = 0$ there exists a gluing $\mathcal{G}'$ of $\text{\normalfont div} K$ and $\text{\normalfont div} L$ so that $\text{\normalfont div}(K\mathcal{G}_\ast L) = (\text{\normalfont div}K)\mathcal{G}'_\ast(\text{\normalfont div}L)$ and
$$r_{d-1}(\mathcal{G}') = f_{d-1}(\text{\normalfont div})\cdot r_{d-1}(\mathcal{G}).$$
\end{proposition}

The $d$-Laplacian operator of the glued complex has the form
$$\Delta(K\mathcal{G}_\ast L) = \begin{pmatrix}
\Delta(K) + D_K & G\\
G^t & \Delta(L) + D_L
\end{pmatrix},$$
where $G$ maps a $d$-face $\tau$ of $K$ to a sum of $d$-faces $\tau'$ of $L$ (with some signs given by orientations) if there are $\sigma\in F_{d-1}(\tau)$ and $\sigma'\in F_{d-1}(\tau')$ such that $\sigma\sim_\mathcal{G}\sigma'$ and $D_K$, $D_L$ are diagonal matrices counting the $(d-1)$-faces for every $d$-face which are involved in gluing for $K$ and $L$, respectively. Thus if we denote by $D$ the maximal down-degree of $K\mathcal{G}_\ast L$ we have
$$||D_K||_{L^1}, ||D_L||_{L^1} \leq D\cdot \max(f_d(K), f_d(L))$$
and
$$||G||_{L^1} \leq r_{d-1}(\mathcal{G}).$$

\section{The Universal Limit Theorem for inclusion-uniform subdivisions}\label{sec:universality}

Now that we have all relevant notions from the introductory section at hand we can prove the main result of this paper, Theorem \ref{th:universality}.

The proof works in two steps which we will state in two propositions. The theorem then follows from the combination of Propositions \ref{prop:dominance} and \ref{prop:convergence}.

For the rest of the chapter let $d$ and $\text{div}$ as in Theorem \ref{th:universality} be fixed. Note that the non-triviality of $\text{div}$ can be equivalently states as $f_d(\text{div}) > 1$. Further let $K$ be an arbitrary initial $d$-dimensional complex. 
$(K_n)_{n\in\mathbb{N}}$ denotes the sequence of complexes generated by iterated application of $\text{div}$ to the initial complex $K$, i.e. $K_n := \text{div}^n K = \text{div}K_{n-1}$, $K_0 = K$. Furthermore by $\mathcal{L}_n$ and $\Lambda_n$ we denote the corresponding sequence of Laplacians and their shifted spectral quantile functions $\Lambda(\mathcal{L}_n)\in L^1([0,1])$, respectively. The claim is thus that $\Lambda_n$ converges towards a universal distribution of eigenvalues depending only on $d$.

\begin{proposition}[Dominance of local spectra]\label{prop:dominance}
Let $\Delta_d$ denote the standard-$d$-simplex. Then in the setting of Theorem \ref{th:universality} it holds that
$$||\Lambda_n - \Lambda(\mathcal{L}(\text{\normalfont div}^n \Delta_d))||_{L^1} \xrightarrow{n\rightarrow\infty} 0,$$
i.e. the spectral quantile function of $K_n$ is asymptotically $L^1$-equivalent to the spectral quantile function of the sequence obtained by subdividing $\Delta_d$.
\end{proposition}
What this means is that global features of the spectrum eventually become dominated by the local features introduced by subdivision of a single simplex.

\begin{proof}
The proof esentially uses Corollary \ref{cor:wielandt} with a counting of non-zero entries which have to be removed in order to transform $\mathcal{L}_n$ in a suitable block-diagonal form. This counting is mainly performed by Proposition \ref{prop:gluing}.

As $K$ is $d$-dimensional the only faces relevant for $\mathcal{L}_n$ are the faces in $F_d(K_n)$ and their down-adjacencies (with respect to an arbitrary orientation of $K$).
Thus we can without loss of generality assume $K$ to be pure and consequently $K_n$ to be pure aswell.

Let $N := f_d(K)$. Note that $K$ can be written as a gluing of $N$ standard-$d$-simplices by purity;
$$K = \mathcal{G}_\ast(\Delta_d,...,\Delta_d),$$
where $\mathcal{G}$ is defined by the lower-adjecencies of the facets of $K$ und some arbitrary identification with the $N$ copies of $\Delta_d$. In particular $r_d(\mathcal{G}) = 0$.

Since $\text{div}$ is inclusion-uniform the process of subdividing $K$ corresponds to subdividing the copies of $\Delta_d$ according to its subdivision scheme $\{\text{div}\Delta_i\}_{i\in\mathbb{N}}$ under induced identification of their faces so that by iterated application of Proposition \ref{prop:gluing} we can write $K_n$ as
$$K_n = \mathcal{G}^{(n)}(\text{div}^n \Delta_d,...,\text{div}^n\Delta_d).$$
Where the number of identifications of $(d-1)$-faces is
$$r_{d-1}(\mathcal{G}^{(n)}) = (f_{d-1}(\text{div}))^n r_{d-1}(\mathcal{G}).$$
Let $L_n$ denote the sequence of Laplacians of $\text{div}^n\Delta_d$. Then the $d$-Laplacian of $K_n$ is of the form
$$\mathcal{L}_n = \begin{pmatrix}
L_n + D_1 & G_{12} & G_{13} & ... & G_{1N}\\
G_{12}^t & L_n + D_2 & G_{23} & ... & G_{2N}\\
\vdots & \ddots & \ddots & \ddots & \vdots\\
\vdots & & \ddots & \ddots & G_{(N-1)N}\\
G_{1N}^t & ... & G_{(N-2)N}^t & G_{(N-1)N}^t & L_n + D_N
\end{pmatrix},$$
where $D_k$ corrects the degrees on the diagonal of $L_n$ along the boundary of the $k$-th copy of $\text{div}^n\Delta_d$. This correction consists of addition by one for every $(d-1)$-face of a $d$-face involved in the gluing process defined by $\mathcal{G}^{(n)}$. Let $D$ be the maximal down-degree of the facets of $K$, then
$$||D_i||_1 \leq D \cdot (f_{d-1}(\text{div}))^n.$$

Further $G_{ij}$ are the matrices containing the down-adjecencies added by gluing the copies $\text{div}^n\Delta_d$ according to $\mathcal{G}^{(n)}$. Note that only $r_{d-1}(\mathcal{G})$ of those $G_{ij}$ are non-zero matrices and the non-zero $G_{ij}$'s have
$$||G_{ij}||_1 \leq (f_{d-1}(\text{div}))^n$$
so that in total by Corollary \ref{cor:wielandt} we have
$$||\Lambda(\mathcal{L}(\tilde{K}_n)) - \Lambda(\mathcal{L}(K_n))||_{L_1} \leq \frac{(ND+2\cdot r_{d-1}(\mathcal{G}))(f_{d-1}(\text{div}))^n}{N\cdot (f_d(\text{div}))^n} \leq (D+2\cdot r_{d-1}(\mathcal{G}))\Big(\frac{f_{d-1}(\text{\normalfont div})}{f_d(\text{\normalfont div})}\Big)^n.$$
where
$$\tilde{K}_n = \bigsqcup_{j=1}^N \text{div}^n \Delta_d$$
with Laplacian matrix
$$\mathcal{L}(\tilde{K}_n) = \text{diag}(L_n,...,L_n).$$
Note that by this equation it holds that
$$\Lambda(\tilde{K}_n) = \Lambda(\text{div}^n\Delta_d).$$
Thus the claim holds iff $$f_{d-1}(\text{div}) < f_d(\text{div}).$$
This will be shown in Lemma \ref{lemm:ineqfvecs}.
\end{proof}

The above proposition immediately shows universality of a limiting function if it exists. The following proposition shows its existence.

\begin{proposition}[Convergence of local spectra]\label{prop:convergence}
Let $K = \Delta_d$ in the setting of Theorem \ref{th:universality}. Then the sequence $(\Lambda_n)_{n\in\mathbb{N}}$ converges in $L^1$.
\end{proposition}
\begin{proof}
To this end we show that $(\Lambda_n)_{n\in\mathbb{N}}$ is a Cauchy sequence - showing existence of a limit by completeness of $L^1$.

The sequence $K_n$ in this case can be obtained as $K_0 = \Delta_d$ and
$$K_n = \text{div}^{n-1} (\text{div}\Delta_d).$$

Note that
$$\text{div}\Delta_d = \mathcal{G}(\Delta_d,...,\Delta_d)$$
where $\mathcal{G}$ glues $f_d(\text{div})$-many $d$-faces along at most $\frac{d+1}{2} f_{d}(\text{div})$ $(d-1)$-faces (note that $\text{div}\Delta_d$ has to be a pseudo-manifold as a triangulation of the $d$-disk), i.e.
$$r_{d-1}(\mathcal{G}) \leq \frac{d+1}{2} f_d(\text{div})$$
and
$$r_d(\mathcal{G}) = 0.$$
Thus as in the above proposition we have
$$||\Lambda_n - \underbrace{\Lambda(\bigsqcup_{i=1}^{f_d(\text{div})} K_{n-1})}_{= \Lambda_{n-1}}||_{L^1} \leq \underbrace{\frac{(d+1+\frac{d+1}{2}f_d(\text{div}))}{f_d(\text{div})}}_{=:c} \Big(\frac{f_{d-1}(\text{div})}{f_d(\text{div})}\Big)^n.$$
We denote by
$$q_d(\text{div}) := \frac{f_{d-1}(\text{div})}{f_d(\text{div})}$$
and will obtain from Lemma \ref{lemm:ineqfvecs}
$$q_d(\text{div}) < 1.$$
Denote by $n_m = f_d(K_m) / f_d(K_n) = f_d(\text{div})^{m - n}$. Applying the above inequality $m-n$ times, $m > n$, we obtain by triangle inequality that 
$$||\Lambda_m - \underbrace{\Lambda(\bigsqcup_{i=1}^{n_m} K_n)}_{=\Lambda_n}||_{L^1} \leq c\sum_{i=n+1}^\infty q_d(\text{div})^i \xrightarrow{n\rightarrow\infty} 0$$
where the right-hand side is a cut-off of a convergent geometric series. 
Thus $(\Lambda_n)_{n\in\mathbb{N}}$ is a Cauchy sequence. By the completeness of $L^1([0,1])$ we obtain the claim.
\end{proof}

\begin{lemma}\label{lemm:ineqfvecs}
Let $\text{\normalfont div}$ be a non-trivial inclusion-uniform subdivision and $q_d(\text{\normalfont div}) := \frac{f_{d-1}(\text{\normalfont div})}{f_d(\text{\normalfont div})}$. Then it holds that
$$q_d(\text{\normalfont div}) < 1.$$
\end{lemma}
\begin{proof}
Assume that $F_d(\text{div}\Delta_d)$ has less or equal the amount of elements of $F_{d-1}(\text{div}\Delta_{d-1})$. Recall that $\text{div}\Delta_d$ is a subdivision of $\Delta_d$ the standard $d$-simplex with boundary $F_{d-1}(\Delta_d) = \{\sigma_0,...,\sigma_d\}$. By definition every restriction $\text{div}_{\Delta_d}(\sigma_i)$ of $\text{div}\Delta_d$ onto $\sigma_i$ results in a complex isomorphic to $\text{div}\Delta_{d-1}$. But by definition every $\tau\in F_{d-1}(\text{div}_{\Delta_d}\sigma_i)$ is contained in $\sigma_i$ and thus must be contained in the boundary of some $d$-face of $\text{div} \Delta_d$ as this complex is homeomorphic to a $d$-ball. In particular every facet $\tau$ of $\text{div}_{\Delta_d}\sigma_i$ is contained in a unique facet $\sigma_\tau\in \text{div}\Delta_d$. Obviously the strict inequality is thus false and equality would need to hold.
Assume thus that $f_d(\text{div}\Delta_d) = f_{d-1}(\text{div}\Delta_{d-1})$.
Note that it is impossible for $\sigma\in F_d(\text{div}\Delta_d)$ to contain two $(d-1)$-faces in the same $\sigma_i$. This is immediate as by definition a inclusion-uniform subdivision has to be geometric and thus every face $\sigma\in F_{d}(\text{div}\Delta_d)$ has to be spanned by $(d+1)$ affinely independent points. In case two codimension-$1$-faces of $\sigma$ are contained in the same $\sigma_i$ all vertices of $\sigma$ would be contained in $(d-1)$-dimensional convex hull. A contradiction.

Thus it immediately follows from the above that
$$\# \underbrace{\{\sigma_\tau ~|~ i\in\{0,...,d\}, ~ \tau \in F_{d-1}(\text{div}_{\Delta_d}\sigma_i)\}}_{=: M \subseteq F_d(\text{div}\Delta_d)} = f_{d-1}(\text{div}\Delta_{d-1})$$
because if we had a single unmatched simplex $\sigma\in F_d(\text{div}\Delta_d) \setminus M$ we had
$$f_{d-1}(\text{div}\Delta_{d-1}) \leq \# M < f_d(\text{div}\Delta_d)$$
which contradicts the equality we assumed. Further since it is impossible for $\sigma\in F_d(\text{div}\Delta_d)$ to contain two $(d-1)$-faces in the same $\sigma_i$ every $\sigma\in F_d(\text{div}\Delta_d)$ needs to be matched by faces $\tau_i\in F_{d-1}(\text{div}_{\Delta_d}\sigma_i)$, $i=0,...,d$. However, the only $d$-simplex $\sigma\subset\Delta_d$ sufficing $\partial\sigma\cap\mathring{\Delta}_d = \emptyset$ is the full simplex itself. Thus $\text{div}\Delta_d \cong \Delta_d$ and the subdivision is trivial. A contradiction to non-triviality of $\text{div}$.
\end{proof}

\section{Universal Limits of Cone Subdivision}\label{sec:cone}
The following section is devoted to the calculation of an explicit universal limit of an example of finite ramification, i.e. a inclusion-uniform subdivision $\text{div}$ such that $f_{d-1}(\text{div}) = 1$. This property will prove to be convenient in the application of the following method since self-similarity will appear only in one block of our target matrix.

Let $d$ be a given dimension. In the following we calculate the renormalization map for the Cone subdivision which is a special case of \textit{finitely ramified} subdivisions.

Let $K$ be a simplicial complex and for every $\sigma\in F_d(K)$ let $v_\sigma$ denote its barycenter.
The cone subdivision $\text{cd}K$ of $K$ is given by adding to $K^{(d-1)}$ the cone $v_\sigma\ast \partial\sigma$ for every $\sigma\in F_d(K)$. Here $K^{(d-1)}$ denotes the $(d-1)$-skeleton of $K$.

\begin{theorem}\label{th:conespectrum}
Let $d > 1$ and $\mathcal{P}_i$ and $\mathcal{Q}_i$ be the sequences recursively obtained as
$$\mathcal{P}_i := f^{-i}(d+1), \quad \quad \mathcal{Q}_i := f^{-i}(d+3)$$
for the polynomial
$$f(\zeta) = \zeta(d+3-\zeta).$$
Then $\{\mathcal{P}_i, \mathcal{Q}_i ~|~ i\in\mathbb{N}\}$ are mutually disjoint and the universal limit $\Lambda^{\text{\normalfont (cd)}}_d$ is the unique increasing step function on $[0,1]$ attaining values in
$$\bigcup_{i=0}^\infty \mathcal{P}_i\cup\bigcup_{j=0}^\infty \mathcal{Q}_j$$
such that $x\in\mathcal{P}_i\cup\mathcal{Q}_{i}$ is attained on an interval of length
$$\frac{d-1}{2(d+1)^{i+1}}.$$
\end{theorem}

\begin{figure}[h]
    \begin{subfigure}[t]{.3\textwidth}
          \centering
          \includegraphics[width=\linewidth]{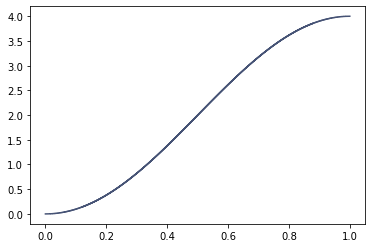}
          \caption{One dimensional limiting distribution. Note that as $\text{cd}$ coincides with the barycentric subdivision $\text{sd}$ for graphs, i.e. $d=1$, and the top-dimensional Laplacian is the $1$-down Laplacian in this case the limiting distribution $\Lambda_1^{\text{(cd)}}(x) = 4\sin^2(\pi x/2)$ as shown in \cite{knill}.}
    \end{subfigure}%
    \hfill
    \begin{subfigure}[t]{.3\textwidth}
          \centering
          \includegraphics[width=\linewidth]{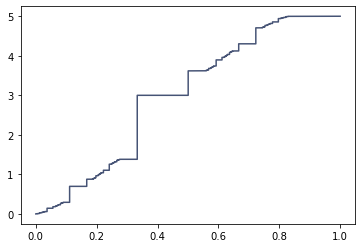}
          \caption{For the two-dimensional limiting distribution note that the continuity of the one-dimensional case does not hold anymore as $\Lambda_2^{\text{(cd)}}$ is a step function (Theorem \ref{th:conespectrum}). }
    \end{subfigure}%
    \hfill
    \begin{subfigure}[t]{.3\textwidth}
          \centering
          \includegraphics[width=\linewidth]{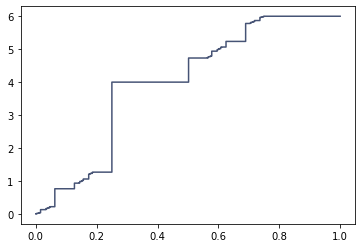}
          \caption{For $d=3$ and higher values of $d$ the steps of early eigenvalues tend to become larger while the decrease in step length of later eigenvalues enhances (cf. the step lengths of eigenvalues $d+1$ and $d-1$, i.e. $\mathcal{P}_0$ and $\mathcal{Q}_0$).}
    \end{subfigure}%
    \\
    \begin{subfigure}[t]{.3\textwidth}
          \centering
          \includegraphics[width=\linewidth]{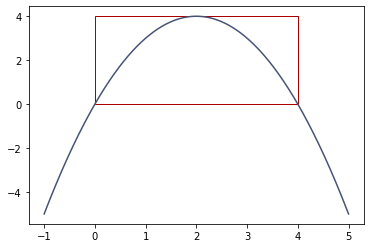}
          \caption{$f(\zeta) = \zeta(4-\zeta)$.}
    \end{subfigure}%
    \hfill
    \begin{subfigure}[t]{.3\textwidth}
          \centering
          \includegraphics[width=\linewidth]{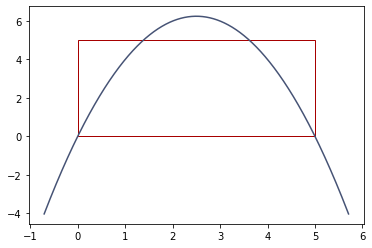}
          \caption{$f(\zeta) = \zeta(5-\zeta)$.}
    \end{subfigure}%
    \hfill
    \begin{subfigure}[t]{.3\textwidth}
          \centering
          \includegraphics[width=\linewidth]{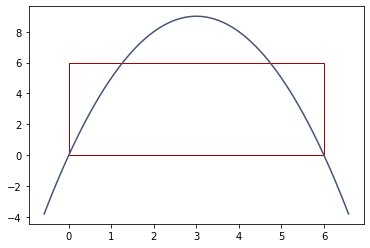}
          \caption{$f(\zeta) = \zeta(6-\zeta)$.}
    \end{subfigure}%
    \caption{Limiting distributions $\Lambda_d^{\text{(cd)}}$ for $d=1,2,3$. Beneath each limit there is a plot of the polynomial $f$ generating the self-similarity of the distributions. The red rectangle shows the range of the feasible values of elements in $\mathcal{P}_i$ and $\mathcal{Q}_i$.}
    \label{fig:limit_plot}
\end{figure}

Note that this theorem encodes information about spectral gaps of the limiting distributions (i.e. ranges in which the total number of eigenvalues vanishes compared to the total number of eigenvalues under $\text{cd}$). We can deduce such gaps from the polynomials $f(\zeta) = \zeta(d+3-\zeta)$ as plotted in Figure \ref{fig:limit_plot}. Note that values in the range $f^{-1}([0, d+3]^c)$ are never obtained as a preimage of a value in $\mathcal{P}_i$ or $\mathcal{Q}_i$ under $f$ since $\mathcal{P}_i\cup\mathcal{Q}_i \subset [0, d+3]$.
Thus whenever $f$ leaves the range $[0, d+3]$ inside the interval $[0, d+3]$ those values can't be obtained in recursion anymore. Same holds true for the complete backwards orbit of this range under $f$ thus inducing gaps in $\Lambda_d^{\text{(cd)}}$ for precisely these ranges.

We show Theorem \ref{th:conespectrum} by representing $\mathcal{L}_n$ (up to the degrees on the diagonal) as the adjacency operator on the $d$-dual graph of $\text{cd}^n \Delta_d$ in the following denoted by
$$\Gamma_n := \Gamma^{(d)}(\text{cd}^n \Delta_d).$$
Subsequently we approximate $\Gamma_n$ by a more convenient graph sequence to work with in terms of asymptotics.

\subsection{Schreier graph approximation of $\Gamma_n$}
Let $\Gamma_n$ denote the $d$-dual graph of $\text{cd}^n \Delta_d$ as above. In this section we will show in Proposition \ref{prop:approx} that it is isomorphic to a Schreier graph on the $n$-th level of an action of a particular self-similar group with a slight error. This error is introduced by the Schreier graph approximation; this is due to the fact that Schreier graphs are regular while $\Gamma_n$ has boundary nodes of degree $d$ though the other (interior) nodes have degree $d+1$. Thus in order to approximate $\Gamma_n$ by a sequence of Schreier graphs we introduce loops on the boundary to artificially make the graph $(d+1)$-regular.
Before we state and prove Proposition \ref{prop:approx} we will need a few definitions and constructions.

To this end we quickly introduce notions of self-similar groups as in \cite{grigorchuk3} and \cite{grigorchuk4}.
Our aim is to reformulate the setting by a group $G$ acting on a $k$-ary tree $\pazocal{T}$ so that the Schreier graph of $G$ on the $n$-th level of the tree is isomorphic to $\Gamma_n$. This will prove to be useful since it allows for a recursive block-description of the adjacency operator of $\Gamma_n$ in terms of a representation of the generators of $G$.

Since every $d$-facet of $K$ gets replaced by $(d+1)$ copies of a $d$-simplex under $\text{cd}$ the natural choice is $k = d+1$ and $\pazocal{T}$ is the tree with vertex set $X^\ast$, the words of finite length over the alphabet $X = [d+1]$, with root $\emptyset$ (the empty word) and adjacencies given by right-adjunction of a single symbol, i.e. the word $w$ has children of the form $wx$ for $x\in X$. We will further use the notation $X^\ast$ of the vertex set of $\pazocal{T}$ for $\pazocal{T}$ itself.
Note that by this definition the $n$-th level of $X^\ast$ is the set $X^n$ of words of length $n$ over $X$.

Now in order to obtain a self-similar Schreier graph sequence from $X^n$ we define what a self-similar group is by action on $X^\ast$. To this end consider the group $\text{Aut}(X^\ast)$ of all automorphisms of the $(d+1)$-ary tree $X^\ast$. Its elements are bijections of the set $X^\ast$ onto itself which fix the root $\emptyset$ and preserve adjacency relations. Note that for a vertex $v\in X^n$ on level $n$ the subtree $\pazocal{T}_v$ is isomorphic to $\pazocal{T}$ itself by the $n$-fold left-shift $w_1...w_k \mapsto w_{n+1}...w_k$.
Thus every automorphism $\varphi\in\text{Aut}(X^\ast)$ is given by a permutation $\sigma\in S_X$ of the first level $X^1 = X$ and a tuple of $(d+1)$ elements describing how $\varphi$ acts on the subtrees $\pazocal{T}_v \cong X^\ast$ for each $v\in X^1$, i.e.
$$(\varphi_1,...,\varphi_{d+1})\in \text{Aut}(X^\ast)^{d+1}.$$

We now say that a subgroup $G\leq \text{Aut}(X^\ast)$ is self-similar if for every $\varphi\in G$ the elements $\varphi_1,...,\varphi_{d+1}$ are themselves elements of $G$.

Having a self-similar group $G$ and a finite set of generators $S$ the sequence of Schreier graphs defined by $G$ (with respect to $S$) is given by $G_n := (X^n, E_n)$ where $E_n$ is defined over $S$ by
$$E_n := \{(w,s\cdot w) ~|~ w\in X^n, s\in S\}.$$
Note that in case $\{s^{-1} ~|~ s\in S\} = S$ we obtain an undirected graph. Also observe that if $S$ acts such that for every $w\in X^n$ and $s_1,s_2\in S$ from $s_1\cdot w = s_2\cdot w$ it follows that $s_1 = s_2$ the adjacency matrix of $G_n$ is given by
$$A(G_n) = \sum_{s\in S} \rho(s)$$
for the representation $\rho: G \rightarrow \text{GL}_{|X^n|}(\mathbb{C})$ defined by the action of $G$ on $X^n$ under some identification of $X^n$ with $[|X^n|]$, i.e. let $\iota: X^n \rightarrow [|X^n|]$ be a bijection, then for $\varphi\in G$ let $\rho(\varphi)\cdot e_{\iota(w)} = e_{\iota(\varphi\cdot w)}$. In particular every $\rho(\varphi)$ is a permutation matrix.

Since the graph $\Gamma_n$ to be approximated does not contain loops we introduce the notion of the \textit{reduced Schreier graph $\tilde{G}_n$ defined by $G$ (with respect to $S$)} as the graph $G_n$ with loops removed. We say that $G_n$ approximates a graph sequence $\Gamma_n$ if $\tilde{G}_n$ is isomorphic to $\Gamma_n$ and for $\ell(G_n)$ the number of loops of $G_n$ it holds that
$$\ell(G_n) \ll v(G_n),$$
i.e. $G_n$ is obtained (up to isomorphism) from $\Gamma_n$ by adding an asymptotically small number of loops. Note that the motivation for this notion of approximating sequences is due to Corollary \ref{cor:wielandt} since the addition of loops to $\Gamma_n$ corresponds to the addition or subtraction of $\ell(G_n)$-many ones along the diagonal of $A(\Gamma_n)$ or $\mathcal{L}(\Gamma_n)$, respectively. Thus
$$||\Lambda_{\mathcal{L}(\Gamma_n)} - \Lambda_{\mathcal{L}(G_n)}||_{L^1} \leq \frac{\ell(G_n)}{v(G_n)} \xrightarrow{n\rightarrow\infty} 0$$
so that if we want to describe $\Lambda^{\text{(cd)}}_d$ from Theorem \ref{th:universality} a spectral decimation of $G_n$ suffices which will be more convenient to work with in this manner.

We will now show that the sequence of graphs $\Gamma_n$ is approximated by the Schreier graph sequence $G_n$ generated by the action of the following group $G\leq \text{Aut}(X^\ast)$: First consider the cyclic permutation
$$\alpha = ((d+1)~d~(d-1)~...~2~1) \in S_X$$
and the automorphism $a$ applying $\alpha$ to the last letter of the given word, i.e.
$$a(wx) = w\alpha(x)$$
for $w\in X^{n-1}$, $x\in X$.
Note that $a$ is of order $d+1$ and consider the cyclic group $A$ generated by $a$. Its $n$-th level Schreier graph with respect to $S = \{a,a^2,...,a^d\}$ is the graph consisting of $(d+1)^{n-1}$ disjoint copies of $K_{d+1}$, one for each set of the form
$$\{wx ~|~ x\in X\}$$
with $w\in X^{n-1}$ fixed. The copies of $K_{d+1}$ here correspond to copies of the dual graph of $\text{cd}\Delta_d$. In order to model the adjacencies between these copies we need to introduce another group generator $b$.

Let $b$ be given by the following self-similar description
$$b(wx) = \begin{cases}
a^{d+1-x}(w)\cdot (d+1-x) &, ~ x\neq d+1\\
b(w)\cdot x &, ~ x = d+1
\end{cases}$$
and initial condition $b(i) = i$ for $i\in X$.
Here $\cdot$ denotes the concatenation of a word with a letter. Note that the initial condition includes loops in the Schreier graph $G_n$.

\begin{figure}
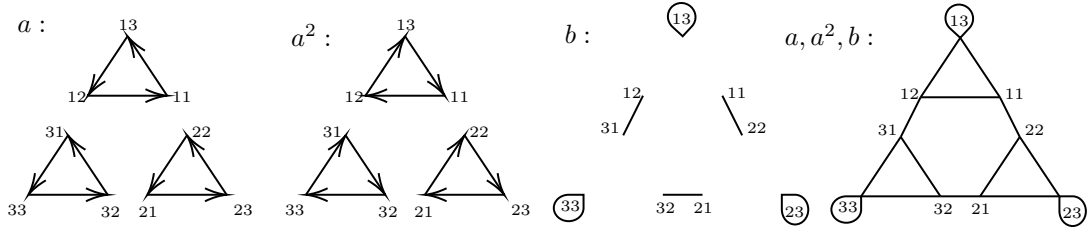

    \centering
    \include{tikz/schreier_decomp}
    \caption{Schreier graphs generated by the choices $\{a\}$, $\{a^2\}$, $\{b\}$ and $\{a,a^2,b\}$ of generators $S$ and the group $G$ generated by $S$.}
    \label{fig:schreier_decomp}
\end{figure}

Let $G$ be the group generated by $a$ and $b$. In order to show that $G_n$ approximates $\Gamma_n$ we analyze the elementary cell of our subdivision sequence (cf. Figure \ref{fig:schreier_decomp} for the case $d=2$ and $n = 3$).
\begin{lemma}
$\Gamma_1 = \Gamma^{(d)}(\text{\normalfont cd}\Delta_d)$ is isomorphic to the complete graph $K_{d+1}$ where the vertices of $K_{d+1}$ are in bijection with the boundary faces $F_{d-1}(\Delta_d)$ over the map $\sigma\mapsto v\ast\sigma$ for $v$ being the barycenter of $\Delta_d$.
\end{lemma}
\begin{proof}
To this end note first that by definition every $d$-face of $\text{cd}\Delta_d$ shares a common $(d-1)$-face with every other $d$-face. This follows from the fact that $\text{cd}\Delta_d$ is defined as the cone over the boundary of the standard-$d$-simplex,
$$\text{cd}\Delta_d = v\ast \partial\Delta_d$$
with $v$ its barycenter. Note that every facet $\sigma\in F_{d-1}(\partial\Delta_d)$ thus corresponds to the unique facet $v\cup\sigma\in F_d(\text{cd}\Delta_d)$ by definition of the cone complex. This correspondence is bijective. Furthermore two facets $v\cup\sigma_1, v\cup\sigma_2\in F_d(\text{cd}\Delta_d)$ share a common $(d-1)$-face iff $\sigma_1$ and $\sigma_2$ share a common $(d-2)$-face. But now every two $(d-1)$-faces of $\partial\Delta_d$ share a common $(d-2)$-face. This is due to the fact that every facet $\sigma\in F_{d-1}(\partial\Delta_d)$ has exactly one opposing vertex $w_\sigma$. Every other facet $\tau$ of $\partial\Delta_d$ can then be obtained as
$$w_\sigma\cup(\sigma\setminus\{w_{\tau}\}).$$
Note that the common $(d-2)$-face of $\tau$ and $\sigma$ then is
$$\sigma\setminus\{w_{\tau}\}.$$
\end{proof}

We will now define a bijection $F_d(\text{cd}^n\Delta_d) \cong X^n$ which will turn out to be a graph isomorphism of $\Gamma_n$ and $\tilde{G}_n$. This bijection can be thought of as an addressing scheme or a labeling of the facets of $\text{cd}^n\Delta_d$.

Obviously the only facet of $\Delta_d = \text{cd}^0\Delta_d$ gets mapped to the empty word $\emptyset$.
Next choose an arbitrary labeling of $F_d(\text{cd}\Delta_d) \cong X$.
Let the labeling $\varphi_{n-1}$ for $\text{cd}^{n-1}\Delta_d$ be defined; let $s: F_d(\text{cd}^n \Delta_d) \rightarrow F_d(\text{cd}^{n-1}\Delta_d)$ be the subdivision map restricted to $d$-faces. Note that under $\text{cd}$ every $\nu\in F_d(\text{cd}^{n-1}\Delta_d)$ gets replaced by $d+1$ new $d$-facets of the form
$$v_\nu\ast \sigma$$
for $\sigma\in F_{d-1}(\nu)$. Further let $p$ denote the parental map on level $n$ in $X^\ast$, i.e.
$$p: X^n \rightarrow X^{n-1}; ~ wx \mapsto w.$$
Given $\tau\in F_d(\text{cd}^n\Delta_d)$ we will define $\varphi_n: F_d(\text{cd}^n\Delta_d) \rightarrow X^n$ such that
\begin{equation} \label{eq:varphi}
p\circ\varphi_n = \varphi_{n-1}\circ s,
\end{equation}
i.e. the $d+1$ children of $\varphi_{n-1}(\nu)$ in $X^\ast$ are identified with the $d+1$ facets added for $\nu\in F_d(\text{cd}^{n-1}\Delta_d)$.
Thus in order to define $\varphi_n$ it suffices to give a bijective map $i_\nu: s^{-1}(\nu) \rightarrow X$. Consider $v_\nu\ast\sigma\in s^{-1}(\nu)$, i.e. $\sigma\in F_{d-1}(\nu)$, then we define $i_\nu$ depending on a variety of cases for $\sigma$:
\begin{itemize}
    \item In case $\sigma$ is boundary, i.e. $\sigma$ has no cofaces besides $\nu$, we set $i_\nu(v_\nu\ast\sigma) = d+1$.
    
    \item Otherwise $\sigma$ has another unique coface $\nu'\in F_d(\text{cd}^{n-1}\Delta_d)$, $\nu'\neq \nu$. Then we have another two cases;
    \begin{itemize}
        \item Either $p\circ\varphi_{n-1}(\nu') = p\circ \varphi_{n-1}(\nu)$ then by equation (\ref{eq:varphi}) there exists $\tau\in F_d(\text{cd}^{n-2}\Delta_d)$ such that
        $$s(\nu) = s(\nu') = \tau.$$
        
        Let $\ell \in \{1,...,d\}$ such that
        $$i_\tau(\nu') \equiv i_\tau(\nu) + \ell ~ (\text{mod}~d+1)$$
        then set
        $$i_\nu(v_\nu\ast\sigma) = \ell.$$
        
        \item or $p\circ \varphi_{n-1}(\nu') \neq p\circ \varphi_{n-1}(\nu)$ then let $i_\nu(v_\nu\ast\sigma) = d+1$.
    \end{itemize}
\end{itemize}
Note that this definition of $i_\nu$ is a well-defined bijection because there is always only one outwards pointing face of every facet, i.e. a face which is either boundary or has another coface which is not a child node of a common facet in $\text{cd}^{n-2}\Delta_d$.
Furthermore when assuming $\nu$ fixed every facet $\nu'$ which shares a $(d-1)$-face with $\nu$ which is not outwards pointing (i.e. $s(\nu) = s(\nu')$) defines a unique value of $\ell$ since $i_\tau$ is a bijection.

\begin{figure}
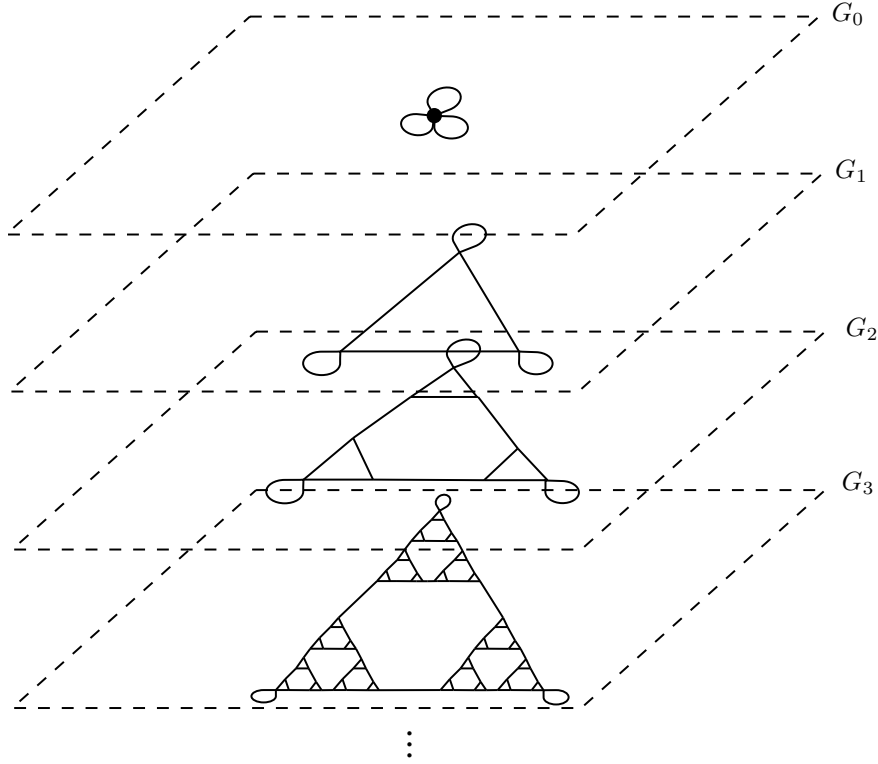

    \centering
    \include{tikz/levels}
    \caption{The Schreier graph approximation $G_n$ for $n \in \{0,1,2\}$ for $d = 2$. Note the structure of the ternary tree indicated by the positions of the triangles $K_3$ under every node of one layer above.}
\end{figure}

\begin{proposition}\label{prop:approx}
$\varphi_n$ defines an isomorphism of the graphs $\Gamma_n$ and $\tilde{G}_n$. Furthermore $G_n$ has $d+1$ loops, i.e. $G_n$ approximates $\Gamma_n$.
\end{proposition}
\begin{proof}
We already know that the map $\varphi_n: F_d(\text{cd}\Delta_d) \rightarrow X^n$ is a bijection. Note that the respective sets are the vertex sets of $\Gamma_n$ and $G_n$, respectively.

Thus in order to obtain an isomorphism we have the show that the edges are in bijection over $\varphi_n$ aswell.

We proceed by induction. For $n = 1$ the claim is obviously true: $\{a,...,a^d\}$ introduces the complete $K_{d+1}\cong \Gamma_1$ in $G_1$ and $b$ acts trivially on $X$ - thus introducing a loop on every vertex in $G_1$. In particular $\varphi_1$ introduces an isomorphism between $\tilde{G}_1$ and $\Gamma_1$.

Now we will show that every edge in $\Gamma_n$ corresponds to the application of $b$ or a power of $a$ on the right-hand side under $\varphi_n$ (up to loops resulting from application of $b$). Note that the edges of $G_n$ are precisely the edges of this form.
Let $\tau\in F_d(\text{cd}^n \Delta_d)$ be given and let
$$\nu := s(\tau)$$
aswell as
$$\tau = v_\nu\ast\sigma$$
for some $\sigma\in F_{d-1}(\nu)$.
Note that by this as mentioned above $\tau$ shares a common $(d-1)$-face with every other face $\tau'\in s^{-1}(\tau)$ of the form
$$\tau' = v_\nu\ast\sigma'$$
for $\sigma'\in F_{d-1}(\nu)$.
Let $i_\nu(\tau)$ and $i_\nu(\tau')$ be as above so that
$$\varphi_n(\tau) = \varphi_{n-1}(\nu)\cdot i_\nu(\tau)$$
and
$$\varphi_n(\tau') = \varphi_{n-1}(\nu)\cdot i_\nu(\tau').$$
Further let $\ell$ be such that
$$i_\nu(\tau') \equiv i_\nu(\tau) + \ell ~(\text{mod} ~ d+1)$$
then by definition of $a$ it is immediate that
$$a^\ell(\varphi_n(\tau)) = \varphi_{n-1}(\nu)\cdot \alpha^\ell(i_\nu(\tau)) = \varphi_{n-1}(\nu)\cdot i_\nu(\tau') = \varphi_n(\tau').$$
Thus the edge
$$(\varphi_n(\tau), \varphi_n(\tau'))$$
is contained in $G_n$ for every $\tau'$. Note also that since for fixed $\tau$ every value of $\ell\in \{1,...,d\}$ occurs for $\tau'$ and thus all edges introduced by action of $a$ in $G_n$ are of this form.

Thus the only other edge incident to $\varphi_n(\tau)$ in $G_n$ is the edge $$(\varphi_n(\tau), b(\varphi_n(\tau))).$$

The only other $(d-1)$-face of $\tau$ which has a coface that is not interior to $\nu$ is $\sigma \leq \tau$. Note that $\sigma$ is itself a $(d-1)$-face of $\nu$ by definition. This $(d-1)$-face is either boundary in which case by definition
$$\varphi_n(\tau) = i(d+1)...(d+1)$$
for arbitrary $i\in X$ and thus $b$ acts on $\varphi_n(\tau)$ as
$$b(\varphi_n(\tau)) = b(i)(d+1)...(d+1)$$
with $b(i) = i$. Thus $b(\varphi_n(\tau)) = \varphi_n(\tau)$ and the corresponding edge in $G_n$ is the loop
$$(\varphi_n(\tau), \varphi_n(\tau))$$
on the boundary face. Note that there are $d+1$-many words of this form $i(d+1)...(d+1)$. Thus $d+1$ loops are included on the boundary faces; those loops are added to $\Gamma_n$ by the transition to $\tilde{\Gamma_n}$.

In case that there is another coface $\tau'$ of $\sigma$ in $\text{cd}^n\Delta_d$ we apply $b$ to $\varphi_n(\tau)$ and need to differentiate between cases in the definition of $b$:

In case $i_\nu(\tau) = (d+1)$ we have
$$b(\varphi_n(\tau)) = b(\varphi_{n-1}(\nu))(d+1)$$
Note by definition of $\varphi_n$ this case corresponds to the case where $\nu$ and $\nu' = s(\tau')$ are not interior to a common $d$-facet in $\text{cd}^{n-2}\Delta_d$. Obviously by symmetry of the fact that $\tau' = v_{\nu'}\ast \sigma$ and $\sigma$ being a face of $\nu'$ and $\nu$ not being interior to a common $d$-facet in $\text{cd}^{n-2}\Delta_d$ we obtain that $i_{\nu'}(\tau') = d+1$ and thus
$$\varphi_n(\tau') = \varphi_{n-1}(\nu')(d+1).$$
But now since $\nu$ and $\nu'$ are not interior to a common $d$-facet of $\text{cd}^{n-2}\Delta_d$ by the induction hypothesis we have
$$\varphi_{n-1}(\nu') = b(\varphi_{n-1}(\nu)$$
and in particular
$$b(\varphi_n(\tau)) = b(\varphi_{n-1}(\nu))\cdot (d+1) = \varphi_{n-1}(\nu')\cdot(d+1) = \varphi_n(\tau').$$
In particular the edge
$$(\varphi_n(\tau), \varphi_n(\tau'))$$
is in $G_n$ and obviously the corresponding edge $(\tau, \tau')$ is in $\Gamma_n$ as $\tau$ and $\tau'$ are $d$-down neighbors.

The last case is when $i_\nu(\tau)\neq d+1$. Again let $\nu' = s(\tau')$.
By definition of $\varphi_n$ we then have a $d$-facet $\mu\in\text{cd}^{n-2}\Delta_d$ such that $\nu$ and $\nu'$ are in the interior of $\mu$. In particular $i_\nu(\tau) = \ell$ where $\ell$ is the unique integer in $\{1,...,d\}$ such that
$$i_\mu(\nu') \equiv i_\mu(\nu) + \ell ~ (\text{mod} ~ d+1).$$
By symmetry of this equation we have
$$i_{\nu'}(\tau') = d+1-\ell$$
in particular. Application of $b$ gives us
$$b(\varphi_n(\tau)) = a^{d+1-\ell}(\varphi_{n-1}(\nu))(d+1-\ell)$$
it thus suffices that $a^{d+1-\ell}(\varphi_{n-1}(\nu)) = \varphi_{n-1}(\nu')$ in order to establish the claim.
This is obvious now; $a^{d+1-\ell}$ acts on $X^{n-1}$ by leaving the first $n-2$ letters fixed and sending the last letter $x$ to the unique representative in $\{1,...,d\}$ of
$$(x + \ell) + (d+1)\mathbb{Z};$$
in particular it sends $i_\mu(\nu)$ onto $i_\mu(\nu')$ and thus
$$a^{d+1-\ell}(\varphi_{n-1}(\nu)) = \varphi_{n-2}(\mu)\cdot i_\mu(\nu') = \varphi_{n-1}(\nu').$$
Thus
$$b(\varphi_n(\tau)) = \varphi_n(\tau')$$
and the edge
$$(\varphi_n(\tau), \varphi_n(\tau'))$$
is contained in $G_n$ as
$$(\varphi_n(\tau), b(\varphi_n(\tau)).$$

\end{proof}

Now we have described the sequence $\Gamma_n$ (up to loops) as a Schreier graph of a self-similar group acting on a $(d+1)$-ary tree in the sense of \cite{grigorchuk}. This viewpoint will be convenient since it gives immediate self-similar descriptions of the Laplacian operator in terms of representations of group elements in a matrix algebra of increasing order.

By all the above it follows that the adjacency matrix of $G_n$ has the form
$$\Xi_n := A(G_n) = \underbrace{\begin{pmatrix}
1 & 1 & ... & 1\\
1 & \ddots &  & \vdots\\
\vdots & & \ddots & 1\\
1 & \dots & 1 & 1
\end{pmatrix}}_{=: J_n} + \underbrace{\begin{pmatrix}
 & & a_{n-1} &\\
 &\reflectbox{$\ddots$} & &\\
a_{n-1}^d & & &\\
& & & b_{n-1}
\end{pmatrix}}_{=: b_n} - {\bf 1}_{(d+1)^n}$$
where $a_n\in M_{(d+1)^n}(\mathbb{C})$ is given as
$$a_n = a_0 \otimes {\bf 1}_{(d+1)^{n-1}}$$
and
$$a_0 := \begin{pmatrix}
0 & 1 & 0 & \dots & 0\\
\vdots & \ddots & \ddots & & \vdots\\
\vdots & & \ddots & \ddots & \vdots\\
0 & & & \ddots & 1 \\
1 & 0 & \dots & \dots & 0
\end{pmatrix}; ~ b_0 := 0 \in M_{d+1}(\mathbb{C});$$
though the initial condition $b_0$ of $b$ is irrelevant for the asymptotic distribution and thus we might also include loops by setting $b_0$ equal to the identity - obtaining the Schreier graph sequence for the hanoi tower group on $3$ pegs in case $d = 2$.

Note that $a_n$ and $b_n$ are the representations of the generators $a$ and $b$ in $\text{GL}_{(d+1)^n}(\mathbb{C})$ as described above. The block structure results from reverse lexicographic ordering, i.e. the $i$-th column and $i$-th row correspond to the words of the form $\ast...\ast i$.

Further we let
$$\Xi_n(\mu, \lambda) = \lambda J_n + b_n - (\lambda + \mu) 1_{(d+1)^n}$$
and
$$D_n(\mu, \lambda) = \det \Xi_n(\mu, \lambda).$$
In particular the map $\mu \mapsto D_n(\mu, 1)$ is the characteristic polynomial of the adjacency matrix $\Xi_n$.

Note that in order to apply Schur-Renormalization we need to determine the determinant of the $d\times d$ upper-left block of $\Xi_n$ which we will denote by $X$ in the following (we drop the subscript $n$ in order to maintain readability).

Note that we have
$$D_n(\mu, \lambda) = \det X \cdot \det(b - \mu {\bf 1}_{(d+1)^{n-1}} - \lambda^2 \Gamma_X(\mu, \lambda)),$$
where $\Gamma_X(\mu, \lambda)$ denotes the block-coronal of $X$ in this case, i.e.
$$\Gamma_X(\mu, \lambda) = \underline{1}_d^t \cdot X^{-1}\underline{1}_d,$$
where
$$\underline{1}_d = (\underbrace{1_{(d+1)^{n-1}},...,1_{(d+1)^{n-1}}}_{d-\text{times}})^t.$$
$n$ will always be inferrable from context.

In order to determine $\Gamma_X$ we will consider $X$ as a matrix over the algebra $\mathcal{A}_n \leq M_{(d+1)^n}(\mathbb{C})$ generated by $a_n$ - which in fact as the group algebra of $C_6$ is a commutative algebra. How this will help us becomes clear in the following sections.

The procedure applied here was developed by Grigorchuk et al. in order to calculate spectra of Schreier graphs associated to groups acting on $k$-ary trees, e.g. in \cite{grigorchuk,grigorchuk2}. We will use the same approach but from a different viewpoint as our starting point is not the group but rather the graph sequence in a self-similar sense.
It is important though that the sequence is representable as a Schreier graph sequence of some group action on the complete $k$-ary tree in order to determine the adjacency matrix in a simple manner.

\subsection{Some elementary properties of the algebra $\mathcal{A}$}
First note that the algebras $\mathcal{A}_n$ are all isomorphic to $\mathcal{A}_0$ via tensoring by $1_{(d+1)^{n-1}}$. Thus we will denote by $\mathcal{A}$ the generic group algebra of $C_6$ commonly realized by $\mathcal{A}_0$. The following results thus also hold in an analogous version over $\mathcal{A}_n$.

\begin{proposition}\label{prop:inversion}
Let $\mu,\lambda$ be given so that
$$x := \mu 1_{d+1} + \lambda\sum_{i=1}^d a^i \in \mathcal{A}$$
is non-singular, then
$$x^{-1} = \frac{1}{(\mu - \lambda)(\mu + d\lambda)}\Big((\mu + (d-1)\lambda) 1_{d+1}  - \lambda\sum_{i=1}^d a^i\Big).$$
\end{proposition}
\begin{proof}
We decompose $x$ as
$$x = (\mu - \lambda) 1_{d+1} + \lambda \underline{1}_{d+1}\cdot \underline{1}_{d+1}^t$$
and apply the Sherman-Morrison formula, Lemma \ref{lemm:SMW}, to yield
$$x^{-1} = \frac{1}{\mu - \lambda} 1_{d+1} - \frac{1}{(\mu - \lambda)^2} \frac{\lambda \underline{1}_{d+1}\cdot\underline{1}_{d+1}^t}{1 + \frac{\lambda}{\mu - \lambda}\underline{1}^t_{d+1}\underline{1}_{d+1}} = \frac{1}{\mu - \lambda} 1_{d+1} - \frac{\lambda}{(\mu -\lambda)(\mu + d\lambda)} \underline{1}_{d+1}\cdot\underline{1}_{d+1}^t.$$
In particular we have
$$x^{-1} = \frac{1}{(\mu - \lambda)(\mu + d\lambda)}\Big((\mu + (d-1)\lambda)1_{d+1} - \lambda \sum_{i=1}^d a^i\Big).$$
\end{proof}

In order to compute determinants of block matrices with blocks in $\mathcal{A}$ we might use the following result relating the usual determinant with the determinant defined in the same way over $\mathcal{A}$, i.e. for $A\in\mathcal{A}^{k\times k}$ let
$$\text{det}_{\mathcal{A}} A := \sum_{\sigma\in S_k} \text{sgn} \sigma A_{1\sigma(1)}...A_{k\sigma(k)} \in\mathcal{A},$$
where $A_{ij} \in \mathcal{A}$ is the block at index $(i,j)$ as usual.

\begin{proposition}[\cite{silvester}]
The usual determinant $\text{\normalfont det}$ factorizes over $\text{\normalfont det}_\mathcal{A}$, i.e. for any $k\times k$ block matrix $A\in\mathcal{A}^{k\times k}$ with blocks in the commutative matrix algebra $\mathcal{A}$ it holds that
$$\text{\normalfont det} A = \text{\normalfont det} ~ \text{\normalfont det}_\mathcal{A} A.$$
\end{proposition}

In particular we also can compute the determinant of $X$ from the beginning of the section as
$$\text{det} \text{det}_{\mathcal{A}} X,$$
where $\det_{\mathcal{A}} X$ in this case is a circulant matrix - of which the determinant is readily calculable by general formulae. A formula of this type needed in the subsequent section will be given by the following lemma.

\begin{lemma}\label{lemm:detx}
$$\det (\mu\cdot 1_{d+1} + \lambda\sum_{i=1}^d a^i) = (\mu + d\lambda)(\mu - \lambda)^d.$$
\end{lemma}
\begin{proof}
This determinant is easily calculated by the Matrix Determinant Lemma, Lemma \ref{lemm:MDL}, after a trivial reparameterization as before;
$$\mu\cdot 1_{d+1} + \lambda\sum_{i=1}^d a^i = (\mu - \lambda)1_{d+1} + \lambda \underline{1}_{d+1}\underline{1}^t_{d+1}$$
so that
$$\det (\mu\cdot 1_{d+1} + \lambda\sum_{i=1}^d a^i) = \Big(1 + \frac{(d+1)\lambda}{\mu - \lambda}\Big)\det((\mu - \lambda)1_{d+1}) = \frac{\mu + d\lambda}{\mu - \lambda} (\mu - \lambda)^{d+1} = (\mu + d\lambda)(\mu - \lambda)^d.$$
\end{proof}

\subsection{Recursion of $D_n$ via renormalization by $\Gamma_X$}
In order to determine the renormalization maps we need to calculate the matrix coronal of $X$ - which will be given by the following lemma.

\begin{lemma}
The block-linear system
$$X\cdot v = \underline{1}_d$$
is solved by
$$v = x^{-1}\odot (a^i + \mu + \lambda)_{i=1}^d\footnote{Here $\odot$ denotes the scalar-multiplication over $\mathcal{A}$; the notation is derived from the common notation $\circ$ for the Hadamard product in combination with $\cdot$ for the scalar multiplication}$$
iff $x = (\mu + \lambda)(\lambda(d-1)-\mu) + 1 + \lambda\sum_{i=1}^d a^i$ is non-singular.
\end{lemma}
\begin{proof}
We just check that
$$X\cdot \tilde{v} = x\underline{1}_d$$
for
$$\tilde{v} = (a^i + \mu + \lambda)_{i=1}^d.$$
In case $d+1$ is even we need to handle the case $j = \frac{d+1}{2}$ for the following seperately. For all other cases we have
\begin{eqnarray*}
(X\cdot \tilde{v})_j &=& -\mu (a^j + \mu + \lambda) + (\lambda + a^j) (a^{d+1 - j} + \mu + \lambda) + \sum_{\substack{i=1\\i\notin\{j, d+1-j\}}}^d \lambda (a^i + \mu + \lambda)\\
&=& -\mu a^j - \mu(\mu + \lambda) + \lambda(a^{d+1-j} + \mu + \lambda) + a^{d+1} + a^j (\mu + \lambda) + \lambda \sum_{\substack{i=1\\i\notin\{j, d+1-j\}}}^d (a^i + \mu + \lambda)\\
&=& -\mu(\mu + \lambda) + a^{d+1} + \lambda a^j + \lambda\sum_{\substack{i=1\\i\neq j}}^d (a^i + \mu + \lambda)\\
&=& -\mu(\mu + \lambda) + 1 + \lambda\sum_{i=1}^d a^i + \lambda (d-1)(\mu + \lambda)\\
&=& (\lambda(d - 1) - \mu)(\mu + \lambda) + \lambda\sum_{i=1}^d a^i + 1 = x.
\end{eqnarray*}
In case $d+1$ is even and $j = \frac{d+1}{2}$ we have
\begin{eqnarray*}
(X\cdot \tilde{v})_j &=& (a^j -\mu)(a^j + \mu + \lambda) + \sum_{\substack{i=1\\i\neq j}} \lambda (a^i + \mu + \lambda)\\
&=& a^{2j} + a^j(\mu + \lambda) - \mu(a^j + \mu + \lambda) + \lambda\sum_{\substack{i=1\\i\neq j}}^d (a^i + \mu + \lambda)\\
&=& 1 + \lambda a^j - \mu (\mu + \lambda) + \lambda \sum_{\substack{i=1\\i\neq j}}^d (a^i + \mu + \lambda)\\
&=& -\mu (\mu + \lambda) + \lambda\sum_{i=1}^d a^i + \lambda (d-1)(\mu+\lambda)\\
&=& (\lambda(d-1) - \mu)(\mu + \lambda) + \lambda \sum_{i=1}^d a^i + 1 = x.
\end{eqnarray*}
Thus in every case we obtain the matrix $x$ and so we have
$$X\tilde{v} = x\odot \underline{1}_d.$$
\end{proof}

In what follows for $x,y\in \mathcal{A}$ we will simply write $\frac{x}{y}$ for $x^{-1}y$ which is a well-defined fraction since $\mathcal{A}$ is a commutative algebra.

\begin{corollary}
We have
$$\Gamma_X(\mu, \lambda) = x^{-1}(\sum_{i=1}^d a^i + d(\mu + \lambda)) = \frac{\sum_{i=1}^d a^i + d(\mu + \lambda)}{\lambda \sum_{i=1}^d a^i + 1 + (\lambda(d-1) - \mu)(\mu + \lambda)}.$$
\end{corollary}

The renormalization is now determined by $\Gamma_X$ over the coefficients of $x^{-1}$.

\begin{proposition}
Let $\alpha := (\mu + \lambda)(\lambda(d-1) -\mu)+1-\lambda$, then we have
$$x^{-1} = \frac{1}{\alpha(\alpha + (d+1)\lambda)}\Big((\alpha + d\lambda) \cdot 1_{(d+1)^{n-1}} - \lambda \cdot \sum_{i=1}^d a^i\Big).$$
In particular the matrix coronal of $X$ is given as
$$\Gamma_X(\mu, \lambda) = \frac{1}{\alpha(\alpha + (d+1)\lambda)} \Big(\big((\alpha + d\lambda)d(\mu + \lambda) - d\lambda \big)1_{(d+1)^{n-1}} + \big(\alpha + \lambda - d\lambda (\mu + \lambda)\big)\sum_{i=1}^d a^i\Big).$$
\end{proposition}
\begin{proof}
The formula of $x^{-1}$ can easily be inferred from Proposition \ref{prop:inversion}.

Furthermore note that
\begin{eqnarray*}
\Gamma_X(\mu, \lambda) &=& \frac{1}{\alpha(\alpha + (d+1)\lambda)}((\alpha + d\lambda) 1_{(d+1)^{n-1}} - \lambda\sum_{i=1}^d a^i)(\sum_{i=1}^d a^i + d(\mu + \lambda))\\
&=& \frac{1}{\alpha(\alpha + (d+1)\lambda)}(\alpha + d\lambda - \lambda d(\mu + \lambda))\sum_{i=1}^d a^i + (\alpha + d\lambda) d(\mu + \lambda) 1_{(d+1)^{n-1}} - \lambda (\sum_{i=1}^d a^i)^2
\end{eqnarray*}
and
$$(\sum_{i=1}^d a^i)^2 = d1_{(d+1)^{n-1}} + (d-1)\sum_{i=1}^d a^i$$
obtaining the wanted representation of $\Gamma_X(\mu, \lambda)$.
\end{proof}

Now let
$$\mu' = \mu + \frac{\lambda^2}{\alpha(\alpha + (d+1)\lambda)}\big((\alpha + d\lambda) d(\mu + \lambda) - d\lambda\big)$$
and
$$\lambda' = -\frac{\lambda^2}{\alpha(\alpha + (d+1)\lambda)}\big(\alpha + \lambda - d\lambda (\mu + \lambda)\big).$$

\begin{corollary}\label{cor:bidim_spectral}
We have
$$D_n(\mu, \lambda) = \det X \cdot D_{n-1}(\mu', \lambda').$$
\end{corollary}
\begin{proof}
By Schur Renormalization, Lemma \ref{lemm:SR}, we have
\begin{eqnarray*}
D_n(\mu, \lambda) &=& \det X \det(b - \mu\cdot 1_{(d+1)^{n-1}} - \lambda^2\Gamma_X(\mu, \lambda))\\
&=& \det X \det\Big(b - \underbrace{\frac{\lambda^2}{\alpha(\alpha + (d+1)\lambda)}\big(\alpha + \lambda - d\lambda (\mu + \lambda)\big)}_{=\lambda'}\sum_{i=1}^d a^i\\
&&- \Big(\underbrace{\mu + \frac{\lambda^2}{\alpha(\alpha + (d+1)\lambda)}\big((\alpha + d\lambda) d(\mu + \lambda) - d\lambda\big)\Big)}_{= \mu'}1_{(d+1)^{n-1}}\Big)
\end{eqnarray*}
and thus the formula follows from above proposition.
\end{proof}

In order to facilitate computation in what follows we give a factorization of the terms in $\mu'$ and $\lambda'$ from $\Gamma_X(\mu, \lambda)$.
\begin{lemma}
We have
$$\mu' = \mu + \frac{d\lambda^2 ((d-1)\lambda^2 + (d-2) \lambda \mu - \mu^2 + \mu)}{((d-1)\lambda - \mu + 1)((d-1)\lambda^2 + (d-2)\lambda\mu - \lambda - \mu^2 + 1)}$$
and
$$\lambda' = \frac{\lambda^2 (\lambda + \mu - 1)}{((d-1)\lambda - \mu + 1)((d-1)\lambda^2 + (d-2)\lambda\mu - \lambda - \mu^2 +1)}$$
\end{lemma}
\begin{proof}
First expand $\alpha$ as
$$\alpha = (d-1)\lambda^2 + (d-2)\lambda\mu - \lambda -\mu^2 + 1.$$
Observe that in our given factorization this is the final form of this term.
The degree one term in the denominator stems from $\alpha + (d+1)\lambda$ which expands as
$$\alpha + (d+1)\lambda = (d-1)\lambda^2 +(d-2)\lambda\mu + d\lambda - \mu^2 + 1 = (\lambda + \mu + 1)((d-1)\lambda -\mu +1).$$
We will show that both numerators are divisible by $(\lambda + \mu + 1)$ - resulting in this term being cancelled.

Let $s_0$ and $s_1$ denote the numerators of the quotients in $\mu'$ and $\lambda'$ ignoring $\lambda^2$ respectively, i.e.
$$s_0 := (\alpha + d\lambda)d(\mu+\lambda) - d\lambda$$
and
$$s_1 := \alpha + \lambda - d\lambda(\mu + \lambda)$$
with respective expansions
$$s_0 = d\Big((d-1) \lambda^3 + (2d - 3) \lambda^2 \mu + (d - 1) \lambda^2 + (d - 3) \lambda\mu^2 + (d - 1)\lambda\mu - \mu^3 + \mu\Big)$$
$$s_1 = -\lambda^2 - 2\lambda\mu - \mu^2 + 1.$$
From here the claimed factorization
$$s_0 = d(\lambda + \mu + 1)((d-1)\lambda^2 + (d-2)\lambda\mu - \mu^2 + \mu)$$
$$s_1 = -(\lambda + \mu - 1)(\lambda + \mu + 1)$$
is easily verified.
\end{proof}

\subsection{Calculation of $\det X$}
In this subsection we will drop the subscript from $a_0$ and will simply denote it by $a$ as every calculation is performed over $\mathcal{A} = \mathcal{A}_0$.

We decompose $X_0$ as
$$X_0 = \lambda \underline{1}_d\underline{1}_d^t + \underbrace{A - (\mu + \lambda) 1_{d}}_{=: Y}$$
for
$$A = \begin{pmatrix}
& & a\\
& \reflectbox{$\ddots$} & \\
a^d & &
\end{pmatrix}.$$

\begin{proposition}\label{prop:detAX}
We have
$$\text{det}_\mathcal{A} X_0 = (1 + \lambda \Gamma_Y(\mu, \lambda)) \big((\mu + \lambda - 1)(\mu + \lambda + 1)\big)^{\lfloor d/2\rfloor} \cdot \begin{cases}
1 &, ~ d\text{ even}\\
-\mu - \lambda + a^{\frac{d+1}{2}} &, ~ d\text{ odd}
\end{cases}.$$
\end{proposition}
\begin{proof}
First we apply the matrix determinant lemma to obtain
$$\text{det}_\mathcal{A} X_0 = \text{det}_\mathcal{A}(Y + \lambda \underline{1}_d\underline{1}_d^t) = \text{det}_\mathcal{A} Y \cdot \text{det}_\mathcal{A}(1 + \lambda \underline{1}^t_d Y^{-1}\underline{1}_d).$$
Note now that the right-most matrix is a $1\times 1$ block matrix and as such has determinant
$$\text{det}_\mathcal{A}(1 + \lambda \underline{1}_d^t Y^{-1} \underline{1}_d) = 1 + \lambda\Gamma_Y(\mu, \lambda)$$
over $\mathcal{A}$.

We now compute $\text{det}_\mathcal{A} Y$ by bringing $Y$ into upper-triangular form. For $d$ even one upper-triangular form is
$$\begin{pmatrix}
-\mu - \lambda & & & & & a \\
 & \ddots & & & \reflectbox{$\ddots$} \\
 & & -\mu - \lambda & a^{d/2} & &\\
 & & 0 & \beta & &\\
 & \reflectbox{$\ddots$} & & & \ddots & \\
0 & & & & & \beta
\end{pmatrix}$$
by elementary transformations, where $\beta = -\mu - \lambda + \frac{a^{d+1}}{\mu + \lambda} = -\mu - \lambda + \frac{1}{\mu + \lambda}$

For $d$ odd a similar upper-triangular form looks like
$$\begin{pmatrix}
-\mu - \lambda & & & & & & a \\
 & \ddots & & & & \reflectbox{$\ddots$} \\
 & & -\mu - \lambda & & a^{\frac{d-1}{2}} &\\
 & & & -\mu - \lambda + a^{\frac{d+1}{2}} & &\\
 & & 0 & & \beta \\
 & \reflectbox{$\ddots$} & & & & \ddots & \\
0 & & & & & & \beta
\end{pmatrix}.$$
Thus
$$\text{det}_\mathcal{A} Y = \begin{cases}
(-\mu - \lambda)^{d/2}\beta^{d/2} &, ~ d\text{ even}\\
(-\mu - \lambda)^{(d-1)/2} \beta^{(d-1)/2} (-\mu - \lambda + a^{\frac{d+1}{2}}) &, ~ d\text{ odd}
\end{cases}$$
By $\beta = \frac{1 - (\mu + \lambda)^2}{\mu + \lambda}$ we obtain 
$$(-\mu - \lambda) \beta = (\mu +\lambda)^2 - 1 = (\mu + \lambda - 1)(\mu + \lambda + 1)$$
and consequently
$$\text{det}_\mathcal{A} Y = \big((\mu + \lambda -1)(\mu + \lambda + 1)\big)^{\lfloor d/2\rfloor} \begin{cases}
1 &, ~ d\text{ even}\\
-\mu - \lambda + a^{\frac{d+1}{2}} &, ~ d\text{ odd}
\end{cases}.$$
\end{proof}
\begin{lemma}\label{lemm:coronY}
The block-linear system
$$Y\cdot v = \underline{1}_d$$
is solved by
$$v = \frac{1}{1 - (\lambda + \mu)^2} (a^i + \mu + \lambda)_{i=1}^d.$$
In particular it holds that
$$\Gamma_Y(\mu) = \frac{\sum_{i=1}^d a^i + d(\mu + \lambda)}{1 - (\mu + \lambda)^2}.$$
\end{lemma}
\begin{proof}
This fact is again easily checked by calculations. Let $\tilde{v} := (a^i + \mu + \lambda)_{i=1}^d$; for every $i\neq (d+1)/2$ it holds that
$$(Y\cdot \tilde{v})_i = -(\mu + \lambda) (a^i + \mu + \lambda) + a^i (a^{d-i+1} + \mu + \lambda) = 1 - (\mu + \lambda)a^i - (\mu + \lambda)^2 + a^i(\mu + \lambda) = 1 - (\mu + \lambda)^2.$$
In case $d$ is odd and $i = (d+1)/2$ we have
$$(Y\cdot \tilde{v})_i = (a^{(d+1)/2} - (\mu + \lambda))(a^{(d+1)/2} + \mu + \lambda) = 1 - (\mu +\lambda)^2$$
thus showing the claim.
\end{proof}

In order to obtain the determinant of $X_0$ we need to calculate the determinant of $-(\mu + \lambda) + a^{\frac{d+1}{2}}$ for $d$ odd now.
\begin{lemma}
Assume $d$ is odd. Then
$$\det(-(\mu + \lambda) + a^{\frac{d+1}{2}}) = (\mu + \lambda+1)^{\frac{d+1}{2}}(\mu + \lambda - 1)^{\frac{d+1}{2}}.$$
\end{lemma}
\begin{proof}
Similar to the upper-triangular form of $Y$ we might bring this matrix into the upper-triangular form
$$\begin{pmatrix}
-(\mu + \lambda) & & & 1 & & \\
 & \ddots & & & \ddots & \\
 & & -(\mu + \lambda) & & & 1 \\
0 & & & \beta & & \\
 & \ddots & & & \ddots & &\\
 & & 0 & & & \beta &
 \end{pmatrix}$$
for $\beta = -\mu - \lambda + \frac{1}{\mu + \lambda} = \frac{1 - (\mu + \lambda)^2}{\mu + \lambda}$.
Consequently 
$$\det(-(\mu + \lambda) + a^{\frac{d+1}{2}}) = \big(-(\mu + \lambda)\cdot \beta \big)^{\frac{d+1}{2}} = \big((\mu + \lambda - 1)(\mu + \lambda + 1)\big)^{\frac{d+1}{2}}$$
showing the claim.
\end{proof}
Thus we are ready to calculate the determinant of $X_0$.

\begin{proposition}\label{prop:detX}
Let $\phi(\mu, \lambda) := \mu^2 - (d-1)\lambda^2 - (d-2)\lambda\mu - 1$, then
$$\det X_0 = (\mu - (d-1)\lambda - 1)(\mu + \lambda + 1)(\phi(\mu, \lambda) + \lambda)^d ((\mu + \lambda)^2 - 1)^{\binom{d+1}{2} - (d+1)}.$$
\end{proposition}
\begin{proof}
Combining the last two lemmata with Proposition \ref{prop:detAX} we first obtain for $d$ odd
\begin{eqnarray*}
\det \text{det}_\mathcal{A} X_0 &=& \det (1 + \lambda \Gamma_Y(\mu, \lambda)) \big((\mu + \lambda - 1)(\mu + \lambda + 1)\big)^{(d+1)\cdot (d-1)/2} \cdot \big((\mu + \lambda - 1)(\mu + \lambda + 1)\big)^{(d+1)/2}\\
&=& \det(1 + \lambda \Gamma_Y(\mu, \lambda)) \big((\mu + \lambda - 1)(\mu + \lambda + 1)\big)^{\binom{d+1}{2}}\\
&=& \det(1 + \lambda \Gamma_Y(\mu, \lambda)) \big((\mu + \lambda)^2 - 1\big)^{\binom{d+1}{2}},
\end{eqnarray*}
while for $d$ even we can directly infer this equality from the proposition.

Now note that by Lemma \ref{lemm:coronY} the left-most term becomes
$$\det(1 + \lambda\Gamma_Y(\mu, \lambda)) = \frac{1}{(1 - (\mu + \lambda)^2)^{d+1}}\det(1 - (\mu + \lambda)^2 + d\lambda(\mu + \lambda) + \lambda\sum_{i=1}^d a^i).$$
Note that
$$1 - (\mu + \lambda)^2 + d\lambda(\mu + \lambda) = 1 - \mu^2 + (d-2)\lambda\mu + (d-1)\lambda^2 = -\phi(\mu, \lambda).$$
This determinant has been determined in Lemma \ref{lemm:detx} - yielding
$$\text{det}(\phi(\mu, \lambda) + \lambda\sum_{i=1}^d a^i) = (-\phi(\mu, \lambda) + d\lambda)(-\phi(\mu, \lambda) - \lambda)^d.$$

In total we obtain
$$\det\text{det}_\mathcal{A} X_0 = (-1)^{d+1}(-\phi(\mu, \lambda) + d\lambda)(-\phi(\mu, \lambda) - \lambda)^d ((\mu + \lambda)^2 - 1)^{\binom{d+1}{2} - (d+1)}$$
and thus the postulated form follows from the easy to verify factorization
$$\phi(\mu, \lambda) - d\lambda = (\mu - (d-1)\lambda - 1)(\mu + \lambda + 1).$$
\end{proof}

\subsection{Unidimensional Spectral Decimation of $\det \Xi_n$}
In the preceding section we have deduced a spectral connection between subsequent subdivision steps, i.e. from the recursion presented in Corollary \ref{cor:bidim_spectral} we are able to compute a factorization of the complete auxiliary spectrum - which is the set of roots of $D_n(\mu, \lambda)$ in $\mathbb{R}^2$.

This spectral set can be decomposed into hyperbolae as we will see which provides us with a way to also deduce the unidimensional spectral decimation stated in Theorem \ref{th:conespectrum}. So what we will show now is that the same procedure as in \cite{grigorchuk} is applicable to the case of arbitrary $d$, i.e. the coefficient changes and the term $(d-2)\lambda\mu$ which will appears in $\mu'$ for $d > 2$ does not form an obstruction to spectral decimation.

The main tool for the deduction of unidimensional spectral decimation for $d=2$ in \cite{grigorchuk} is semi-conjugacy of the renormalization $F$ to $f:\mathbb{R}\rightarrow\mathbb{R}; ~ \zeta\mapsto \zeta^2 - \zeta - 3$. Semi-conjugacy means that there is a suitable way to map $\mathbb{R}^2$ to $\mathbb{R}$ so that this parameter mapping $\Psi$ identifies $F$ with $f$.

This semi-conjugacy will in fact stay intact for $d > 2$, though for the unidimensional map
$$f(\zeta) = \zeta^2 - (d-1)\zeta - (d+1)$$
and the parameter mapping
$$\Psi(\mu, \lambda) := \frac{\mu^2 - 1 - (d-1)\lambda\mu - d\lambda^2}{\lambda} =: \frac{\Phi(\mu, \lambda)}{\lambda}$$
as will be shown in the following lemma.

\begin{lemma}
$F$ is semi-conjugate to $f$ over $\Psi$, i.e.
$$\Psi\circ F = f\circ \Psi.$$
\end{lemma}
\begin{proof}
This claim can be verified by high-school algebra.
\end{proof}

Now in order to obtain spectral decimation we only need to analyze the behaviour of the factors of $\det X_0$ under renormalization $F$. To this end we will use the semi-conjugacy of $F$ to the $1$-dimensional map $f$.

Analogously to \cite{grigorchuk} let
$$\Phi_\theta(\mu, \lambda) := \Phi(\mu, \lambda) - \theta \lambda = \lambda(\Psi(\mu, \lambda) - \theta) = \mu^2 - 1 - (d-1)\lambda\mu - d\lambda^2 - \theta \lambda$$
$$L(\mu, \lambda) = \mu - (d-1)\lambda - 1$$
$$K(\mu, \lambda) = \phi(\mu, \lambda) + \lambda = \mu^2 - (d-1)\lambda^2 - (d-2)\lambda\mu + \lambda - 1$$
$$A_1(\mu, \lambda) = \lambda + \mu - 1$$
so that
$$\lambda' = \frac{\lambda^2 A_1(\mu, \lambda)}{L(\mu, \lambda)K(\mu, \lambda)}.$$
Then the semi-conjugacy gives us the following lemma as in \cite{grigorchuk}.

\begin{lemma}
Let $\theta\in [-2,d+1]$ and $\theta_0, \theta_1$ be the two distinct real roots of $f(x) - \theta$. Then we have
$$\frac{A_1}{LK}\Phi_{\theta_0} \Phi_{\theta_1} = \Phi_\theta\circ F.$$
\end{lemma}
\begin{proof}
It holds that
$$\Phi_{\theta}\circ F = \lambda'(\Psi\circ F - \theta) = \lambda'(f\circ \Psi - \theta) = \frac{\lambda^2 A_1}{LK}(\Psi - \theta_0)(\Psi - \theta_1) = \frac{A_1}{LK}\Phi_{\theta_0}\Phi_{\theta_1}.$$
\end{proof}

One last thing that remains to calculate is the initial polynomial $D_1(\mu, \lambda)$. Note that
$$\Xi_1(\mu, \lambda) = (-\mu + 1)1_{d+1} + \lambda\sum_{i=1}^d a^i$$
so that
$$D_1(\mu, \lambda) = -(\mu - 1 - d\lambda)(-\mu+1-\lambda)^d = \underbrace{(-1)^{d+1}(\mu - 1 - d\lambda)}_{=: D_0(\mu, \lambda)} A_1^d.$$

Now let
$$A_n(\mu, \lambda) = \begin{cases}
\mu + \lambda - 1 &, ~ n = 1\\
\prod_{\theta\in f^{-(n-2)}(0)} \Phi_\theta &, n > 1
\end{cases}$$
$$B_n(\mu, \lambda) = \begin{cases}
\mu + \lambda + 1 &, n = 2\\
\prod_{\theta\in f^{-(n-3)}(-2)} \Phi_\theta &, n > 2
\end{cases}$$
so that by Proposition \ref{prop:detX} it holds true that
$$\det X_0 = L B_2 K^d (A_1B_2)^{\binom{d+1}{2} - (d+1)}.$$

Note here that the quadratic maps $\Phi_\theta$ are defining the hyperbolae in which the bidimensional auxiliary spectrum can be decomposed, i.e. in which $D_n$ is factorizable (see Proposition \ref{prop:factorcone}). Now in order to factorize $D_n$ into $A_i$'s and $B_i$'s we need to state the behaviour of the factors of $\det X_0$ under composition with $F$.

\begin{lemma}
The following relations with respect to $F$ hold:
\begin{itemize}
    \item[] $$D_0\circ F = \frac{D_0}{L} A_1$$
    \item[] $$A_1\circ F = \frac{A_1}{K} A_2$$
    \item[] For $n\geq 2$:
    $$A_n\circ F = \Big(\frac{A_1}{LK}\Big)^{2^{n-2}}A_{n+1}$$
    \item[] $$B_2 \circ F = \frac{B_2}{K} B_3$$
    \item[] For $n\geq 3$:
    $$B_n\circ F = \Big(\frac{A_1}{LK}\Big)^{2^{n-3}} B_{n+1}$$
\end{itemize}
\end{lemma}
\begin{proof}
The above lemma allows us to show the claims involving $n$ directly as
$$A_n\circ F = \prod_{\theta\in f^{-(n-2)}(0)} \Phi_\theta\circ F = \Big(\frac{A_1}{LK}\Big)^{2^{n-2}} \prod_{\theta\in f^{-(n-1)}(0)} \Phi_\theta = \Big(\frac{A_1}{LK}\Big)^{2^{n-2}} A_{n+1}.$$
The respective claim for $B$ can be shown in a similar fashion. The claims not involving $n$ can again be verified by high-school algebra.
It should be noted that $A_2 = \Phi$, $B_3 = \Phi + 2 \lambda$
\end{proof}

\begin{proposition}\label{prop:factorcone}
The determinant $D_n(\mu, \lambda)$ factorizes as
$$D_1 = D_0 A_1^d$$
$$D_n = D_0 A_1^{\alpha_n}...A_n^{\alpha_1}B_2^{\beta_n}...B_n^{\beta_2}$$
for $n \geq 2$, where the sequences $(\alpha_n)_{n\geq 1}$, $(\beta_n)_{n\geq 2}$ are given by
$$\alpha_n = \beta_n + d, \quad\quad \beta_n = \frac{d-1}{2}((d+1)^{n-1}-1)$$
for $n \geq 2$ and $\alpha_1 = d$.
\end{proposition}
\begin{proof}
For $n = 1$ we have shown the factorization of $D_1$ above.

We now proceed by induction. In order to ease notation we denote for a function $g(\mu, \lambda)$ the renormalized function $g\circ F$ by $g'$.

Let $n \geq 2$ and assume the factorization holds for $n-1$. By the above we have
\begin{eqnarray*}
D_n &=& \Big(LB_2 K^d (A_1B_2)^{\binom{d+1}{2}-(d+1)} \Big)^{(d+1)^{n-2}} D_{n-1}'\\
&=& \Big(LB_2 K^d (A_1B_2)^{\binom{d+1}{2}-(d+1)} \Big)^{(d+1)^{n-2}} D_0'\cdot (A_1')^{\alpha_{n-1}}\cdot...\cdot (A_{n-1}')^{\alpha_1} \cdot (B_2')^{\beta_{n-1}} \cdot ... \cdot (B_{n-1}')^{\beta_2}\\
&=&  \Big(LB_2 K^d (A_1B_2)^{\binom{d+1}{2}-(d+1)} \Big)^{(d+1)^{n-2}} \frac{D_0A_1}{L} \Big(\frac{A_1A_2}{K}\Big)^{\alpha_{n-1}} \Big(\frac{A_1}{LK}\Big)^{\sigma_n} A_3^{\alpha_{n-2}}...A_n^{\alpha_1} \Big(\frac{B_2B_3}{K}\Big)^{\beta_{n-1}} B_4^{\beta_{n-2}}...B_n^{\beta_2}
\end{eqnarray*}
where

\begin{equation}\label{eq:sigma}
    \sigma_n = (\alpha_{n-2} + 2 \alpha_{n-3} + ... + 2^{n-3}\alpha_1) + (\beta_{n-2} + 2 \beta_{n-3} + ... + 2^{n-4}\beta_2)
\end{equation}
for $n > 2$ and $\sigma_2 = 0$. Furthermore for ease of notation let $c_d := \binom{d+1}{2} - (d+1)$.

The above equation for $D_n$ implies that the following equations are necessary for the induction to hold:
\begin{equation}\label{eq:a}
    \alpha_n = c_d(d+1)^{n-2} + \sigma_n + \alpha_{n-1} + 1
\end{equation}
\begin{equation}\label{eq:b}
    \beta_n = (c_d + 1)(d+1)^{n-2} + \beta_{n-1}
\end{equation}

The values of $\alpha$, $\beta$ and $\sigma$ are determined in by the subsequent Lemma \ref{lemm:sequences} to be:
$$\beta_n = \frac{c_d+1}{d}((d+1)^{n-1}-1)$$
$$\alpha_n = \beta_n + d$$
$$\sigma_n = (d+1)^{n-2}-1$$
for $n > 2$. Further note that $\frac{c_d + 1}{d} = \frac{d-1}{2}$.

Now having the sequences $\alpha,\beta,\sigma$ at hand we can verify that these are also sufficient for the induction; this is done by showing that all copies of $L$ and $K$ cancel out. So for the $K$'s it must hold true that
$$d(d+1)^{n-2} = \alpha_{n-1} + \sigma_n + \beta_{n-1} = \Big(2\frac{c_d+1}{d} + 1\Big)(d+1)^{n-2} - \Big(2\frac{c_d+1}{d} + 1\Big) + d.$$
Taking into consideration that
$$2\frac{c_d+1}{d}+1 = d$$
we obtain the validity of the claim.

For the $L$'s we have to check
$$(d+1)^{n-2} = 1 + \sigma_n$$
which already has been verified before.

Thus after the proper reordering and canceling of the terms we obtain
$$D_n = D_0 A_1^{\alpha_n}...A_n^{\alpha_1} B_2^{\beta_n}...B_n^{\beta_2}$$
which is the claimed factorization.
\end{proof}

\begin{lemma}\label{lemm:sequences}
Equations (\ref{eq:sigma}), (\ref{eq:a}) and (\ref{eq:b}) together with the initial values
$$\alpha_1 = d, \quad \quad \beta_1 = 0$$
imply that
$$\sigma_n = (d+1)^{n-2} - 1,$$
$$\beta_n = \frac{c_d+1}{d}((d+1)^{n-1}-1),$$
$$\alpha_n = \beta_n + d.$$
\end{lemma}
\begin{proof}
Using geometric series it follows that
$$\beta_n = (c_d + 1) \sum_{i=0}^{n-2} (d+1)^i =  \frac{c_d+1}{d}((d+1)^{n-1}-1).$$

In order to determine $\alpha$ and $\sigma$ we need to perform an intertwined induction on both. We claim that
$$\sigma_n = (d+1)^{n-2} - 1$$
for $n \geq 3$,
$$\sigma_1 = \sigma_2 = 0$$
and
$$\alpha_n = \beta_n + d$$
for $n\geq 1$ and $\beta_1 = 0$.

Obviously those claims are easily verified for $n\leq 2$. We will show the claim for $n > 2$ by induction. Assume the claims hold for $n-1$. By (\ref{eq:sigma}) for $n$ we have
$$\sigma_n = \alpha_{n-2} + \beta_{n-2} + 2 \sigma_{n-1} = (d+1)^{n-3}\Big(2\frac{c_d+1}{d} + 2\Big) + d - \Big(2\frac{c_d+1}{d}+2\Big).$$
Note now by definition of $c_d$ we have
$$2\frac{c_d+1}{d}+2 = 2\frac{c_d+d+1}{d} = d+1$$
whence
$$\sigma_n = (d+1)^{n-2} - 1.$$
Further by (\ref{eq:a}) we obtain
$$\alpha_n = c_d(d+1)^{n-2} + (d+1)^{n-2} + \alpha_{n-1}$$
which is solved under the initial condition $\alpha_1 = d$ just as (\ref{eq:b}) by
$$\alpha_n = \beta_n + d.$$
\end{proof}

The following proposition is the analogon of Theorem \ref{th:conespectrum} for the adjacency spectrum. Theorem \ref{th:conespectrum} will then follow immediately.

\begin{proposition}\label{prop:adjacencyspectrum}
Let $d > 1$ and $\mathcal{A}_i$ and $\mathcal{B}_i$ be the sequences recursively obtained as
$$\mathcal{A}_i := g^{-i}(0), \quad \quad \mathcal{B}_i := g^{-i}(-2)$$
for the polynomial
$$g(\zeta) = \zeta^2 - (d-1)\zeta - (d+1).$$
Then $\{\mathcal{A}_i, \mathcal{B}_i ~|~ i\in\mathbb{N}\}$ are mutually disjoint and the sequence of shifted spectral quantile functions
$$\Lambda(A(G_n)) = \Lambda(\Xi_n)$$
converges to the unique increasing step function $\Lambda$ on $[0,1]$ attaining values in
$$\bigcup_{i=0}^\infty \mathcal{A}_i\cup\bigcup_{j=0}^\infty \mathcal{B}_j$$
such that for $x\in\mathcal{A}_i\cup\mathcal{B}_{i}$ the value $(d+1)-x$ is attained on an interval of length
$$\frac{d-1}{2(d+1)^{i+1}}$$
in $L^1([0,1])$.
\end{proposition}
\begin{proof}
The claims follow immediately from the factorization provided for $D_n(\mu, \lambda)$ in Proposition \ref{prop:factorcone} - note that the eigenvalues of $\Xi_n$ are given by the roots of $D_n(\mu, 1)$. Under the assumption $\lambda = 1$ we obtain the following relations
$$D_0(\mu, 1) = \mu - (d+1),$$
$$A_1(\mu, 1) = \mu,$$
$$B_1(\mu, 1) = \mu + 2,$$
$$\Phi(\mu, 1) = g(\mu),$$
$$\Phi_\theta(\mu, 1) = g(\mu) - \theta.$$
Giving us the full description of the spectral distribution. The multiplicities are given by the exponents with which the factors appear.

Thus by $\mathcal{A}_i = g^{-i}(0)$ and $\mathcal{B}_i = g^{-i}(-2)$ the spectrum of $\Xi_n$ as a set decomposes as
$$\bigcup_{i=0}^{n-1}\mathcal{A}_i \cup \bigcup_{i=0}^{n-2}\mathcal{B}_i.$$
We will show that this is in fact a partition (i.e. $\{\mathcal{A}_i,\mathcal{B}_i ~|~ i\in\mathbb{N}\}$ are mutually disjoint). Furthermore by the above equations the eigenvalues in $\mathcal{A}_i$ and $\mathcal{B}_i$ are precisely the roots of the factors $A_{i+1}$ and $B_{i+2}$ which occur with exponent $\alpha_{n-i}$ and $\beta_{n-i}$ in $D_n$, respectively.

In what follows we will also see that the factors $A_i$ and $B_i$ do not have multiple roots so that the multiplicities of eigenvalues of $\Xi_n$ in $\mathcal{A}_i$ and $\mathcal{B}_i$ are $\alpha_{n-i}$ and $\beta_{n-i}$, respectively.

The mutual disjointness of $\mathcal{A}_i$ with any $\mathcal{B}_j$ can be seen in the following way: First note that $g(c\cdot (d+1)) > c\cdot (d+1)$ for every $c > 1$;
\begin{align*}
    g(c\cdot (d+1)) &= c^2 (d+1)^2 - c(d-1)(d+1) - (d+1)\\
    &= (d+1)(c^2 (d+1) - c(d-1) - 1)\\
    &> (d+1)(c (d+1) - c(d-1) - 1)\\
    &= (d+1)(2c - 1)\\
    &> c(d+1)
\end{align*}
Thus if for given $x$ the sequence $g^n(x)$ surpasses the value $(d+1)$ it must be strictly increasing.

Assume now that there is $x\in \mathcal{A}_i$ such that there is $j\in\mathcal{B}_j$ with $x\in \mathcal{B}_j$. In this case we have
$$g^i(x) = 0$$
and
$$g^j(x) = -2.$$
Thus in case $i < j$ we have found that
$$g^{j-i}(0) = -2$$
and in case $i > j$ we obtain
$$g^{i-j}(-2) = 0.$$

We will exclude both cases by observing the first few elements of the sequence before it is forced to be strictly increasing by the above observation.
Note that
$$g(0) = -(d+1)\neq -2$$
since $d > 1$. Furthermore
$$g^2(0) = f(-(d+1)) = (d+1)(d+1 + d-1 - 1) = (d+1)\underbrace{(2d-1)}_{> 1}$$
thus not attaining the value $-2$.

For the sequence with $x = -2$ we have
$$g(-2) = 4 + 2(d-1) - (d+1) = d+1$$
and
$$g(d+1) = (d+1)(d+1 - (d-1) - 1) = (d+1)$$
thus stabilizing at $d+1$ not attaining $0$.

This way we have seen that $\mathcal{A}_i$ and $\mathcal{B}_j$ must be disjoint. To see the mutual disjointness of the $\mathcal{A}$'s assume there is $x$ such that $x\in\mathcal{A}_i\cap\mathcal{A}_j$ and assume further that $i < j$. Then obviously we have
$$g^{j-i}(0) = 0$$
which we have shown to be false.

For the $\mathcal{B}$'s assume analogously that there is $x\in\mathcal{B}_i\cap\mathcal{B}_j$ such that $i < j$. Then again
$$g^{j-i}(-2) = -2.$$
But this is false since the sequence stabilizes at $d+1$ immediately.
Thus the mutual disjointness of $\{\mathcal{A}_i,\mathcal{B}_i~|~i\in\mathbb{N}\}$ follows.

That $A_i$ and $B_i$ do not have multiple roots can be obtained as follows: Obviously $A_1$ and $B_2$ do not have multiple roots. The roots of $A_{i+1}$ and $B_{i+1}$ are obtained from the roots of $A_i$ and $B_i$, respectively, by taking $g^{-1}$, i.e. if $\lambda$ is a root of $A_i$ it induces two roots of $A_{i+1}$, namely
$$\Bigg\{\frac{d-1}{2} \pm \sqrt{\Big(\frac{d-1}{2}\Big)^2 + (d+1) + \lambda}\Bigg\}.$$
In particular $\mathcal{A}_{i+1}$ does split in two sets
$$\mathcal{A}_{i+1} = \mathcal{A}_{i+1}^+\cup\mathcal{A}_{i+1}^-$$
each in bijection to $\mathcal{A}_i$ over the maps $\frac{d-1}{2} + r$, $\frac{d-1}{2}-r$ for
$$r(x) = \sqrt{\Big(\frac{d-1}{2}\Big)^2 + (d+1) + x}.$$
Note that $r$ is non-negative and $r(x) = 0$ iff $x =  -(1 + ((d+1)/2)^2) < -2$ since $d > 1$. Thus the sets $\mathcal{A}_i^+$ and $\mathcal{A}_i^-$ have to be disjoint and thus no roots of $A_i$ can be multiple.
Analogously we obtain the same for $\mathcal{B}_i$.

Now we will show the convergence of the spectral cdf of $G_n$ towards the function $\Lambda$ claimed to be the limit.

The convergence of $\Lambda_n = \Lambda(A(G_n)) = \Lambda(\Xi_n)$ towards the increasing step function $\Lambda$ with step values in $\bigcup_{i=0}^\infty \mathcal{A}_i\cup\bigcup_{i=1}^\infty \mathcal{B}_i$ and step length 
$$\frac{d-1}{2(d+1)^{i+1}}$$
for values in $\mathcal{A}_i\cup\mathcal{B}_i$ follows from the fact that the step length of eigenvalues $\lambda\in\mathcal{A}_i$ or $\lambda\in\mathcal{B}_i$ in $\Lambda_n$ is given by
$$\frac{\alpha_{n-i}}{(d+1)^n} = \frac{d-1}{2} \frac{(d+1)^{n-i-1}}{(d+1)^n} + \frac{d+1}{2(d+1)^n}\xrightarrow{n\rightarrow\infty} \frac{d-1}{2(d+1)^{i+1}}$$
or
$$\frac{\beta_{n-i}}{(d+1)^n} = \frac{d-1}{2}\frac{(d+1)^{n-i-1}}{(d+1)^n} - \frac{d-1}{2(d+1)^n} \xrightarrow{n\rightarrow\infty} \frac{d-1}{2(d+1)^{i+1}},$$
respectively.

Thus the step lengths of the steps in $\Lambda_n$ converge uniformly towards their respective step length in $\Lambda$ and the values not attained by $\Lambda_n$ are
$$\bigcup_{i=n}^\infty \mathcal{A}_i \cup \bigcup_{i=n-1}^\infty \mathcal{B}_i.$$
They are attained by $\Lambda$ on a joint volume of 
\begin{eqnarray*}
\sum_{i=n}^\infty 2^i \frac{d-1}{2(d+1)^{i+1}} + \sum_{i=n-1}^\infty 2^i \frac{d-1}{2(d+1)^{i+1}} &=& 2^{n-1}\frac{d-1}{2(d+1)^n} + \sum_{i=n}^\infty 2^i \frac{d-1}{(d+1)^{i+1}}\\
&\overset{d\neq 1}{=}& 2^{n-2}\frac{d-1}{(d+1)^n} + \frac{2^n}{(d+1)^{n}}\\
&=& \frac{2^{n-2}(d+3)}{(d+1)^n}.
\end{eqnarray*}
Obviously since $d + 1 > 2$ this volume vanishes asymptotically.

Thus for $\varepsilon > 0$ there exists $n_0\in\mathbb{N}$ such that for every $n\geq n_0$ it holds that
$$\frac{2^{n-2}(d+3)}{(d+1)^n} < \varepsilon$$
so that also
$$\delta := \frac{d+1}{2(d+1)^n} < \frac{d+1}{2^{n-1}(d+3)} \varepsilon.$$

Let such an $n$ be fixed for now; then by the above considerations we might modify $\Lambda$ by setting all steps with values in
$$\bigcup_{i=n}^\infty\mathcal{A}_i \cup \bigcup_{i=n-1}^\infty \mathcal{B}_i$$
to zero, i.e.
$$\tilde{\Lambda}(x) := \Lambda(x)\cdot \boldsymbol{1}_{\Lambda(x)\in F}$$
for
$$F = \bigcup_{i=0}^{n-1}\mathcal{A}_i \cup \bigcup_{i=0}^{n-2}\mathcal{B}_i.$$
Obviously since the absolute values of elements in $\bigcup_{i=0}^\infty\mathcal{A}_i\cup\bigcup_{i=0}^\infty\mathcal{B}_i$ are bounded by $d+1$ from the above we obtain
$$||\Lambda - \tilde{\Lambda}||_{L^1([0,1])} < \varepsilon \cdot (d+1).$$

Note that $\Lambda - \tilde{\Lambda}$ is supported on a set of measure less than $\varepsilon$.

Subsequently we do the same modification for $\Lambda_n$, i.e. $\tilde{\Lambda}_n(x) := \Lambda_n(x) \cdot \boldsymbol{1}_{\Lambda(x) \in F}$ so that both $\tilde{\Lambda}$ and $\tilde{\Lambda}_n$ are zero on $\Lambda^{-1}(F^c)$. 
We obviously have by the same reasoning as for $\Lambda$ that
$$||\tilde{\Lambda}_n - \Lambda_n||_{L^1([0,1])} < \varepsilon(d+1).$$

Thus on $\Lambda^{-1}(F^c)$ $\tilde{\Lambda}$ aswell as $\tilde{\Lambda}_n$ are zero so that we might consider both functions to be step function on $[0,1-\varepsilon)$ (this can be done because we modified $\Lambda$ only on a countable union of intervals from $[0,1]$). We will denote those step functions by $\tilde{\Lambda}$ and $\tilde{\Lambda}_n$ aswell as their $L^1$-distance stay the same under this transition.

Now order the eigenvalues in $F$ as
$$\lambda_1 < ... < \lambda_k$$
and note that
$$k = \sum_{i=0}^{n-1} 2^i + \sum_{i=0}^{n-2}2^i = 2^n + 2^{n-1} - 2.$$
Furthermore denote by $\ell_i$ and $\ell_i'$ the length of the step with value $\lambda_i$ in $\tilde{\Lambda}$ and $\tilde{\Lambda}_n$, respectively. Note that $\ell_i'$ might be $0$ if the entire step of $\lambda_i$ in $\Lambda_n$ was contained in $\Lambda^{-1}(F^c)$.

As both $\tilde{\Lambda}$ and $\tilde{\Lambda}_n$ have the same finite image set $F$ we can bound the $L^1$ distance in terms of their jumps. Note that the discrepancy $\delta_i := |\ell_i - \ell_i'|$ induces a shift in the subsequent steps by $\delta_i$ where the shift is accounted for in $L^1$-distance by the integration of every subsequent jump over an interval of length $\delta_i$.
We decompose $\delta_i$ in two parts;
$$\delta_i \leq \delta + f_i,$$
where $\delta$ is the discrepancy introduced by the differing length of steps coming from $\Lambda$ and $\Lambda_n$ themselves while $f_i$ is introduced by deleting parts of the steps in the process of going over to $\tilde{\Lambda}_n$ from $\Lambda_n$. Since the range on which we delete is of measure $\varepsilon$ we obtain
$$\sum_{i=1}^k f_i = \varepsilon.$$
Thus letting the jump height of the $i$-th jump be $h_i = \lambda_{i+1} - \lambda_i$ we obtain
\begin{eqnarray*}
||\tilde{\Lambda} - \tilde{\Lambda}_n||_{L^1([0,1-\varepsilon))} &\leq& \sum_{i=1}^k \delta_i \sum_{j=i}^{k-1} h_j\\
&\leq& \sum_{i=1}^k \delta \sum_{j=i}^{k-1}h_j + \sum_{i=1}^k f_i \sum_{j=i}^{k-1}h_j.
\end{eqnarray*}
Obviously $\sum_{j=i}^{k-1} h_j \leq d+3$ since $F\subset [-2,d+1]$. Thus using $\sum_{i=1}^k f_i = \varepsilon$
$$||\tilde{\Lambda} - \tilde{\Lambda}_n||_{L^1([0,1-\varepsilon))} \leq (d+3)(k\delta + \varepsilon).$$
This can be bounded using the equations for $\delta$ and $k$ by
$$k\delta < (2^n + 2^{n-1}-2)\frac{d+1}{2^{n-1}(d+3)}\varepsilon < 3 \frac{d+1}{d+3}\varepsilon.$$
Thus
$$||\tilde{\Lambda} - \tilde{\Lambda}_n||_{L^1([0,1-\varepsilon))} < (3(d+1) + d+3)\varepsilon = (4d+6)\varepsilon;$$
showing the claim.
\end{proof}

\begin{proof}[Proof of Theorem \ref{th:conespectrum}]
Theorem \ref{th:conespectrum} immediately follows from Proposition \ref{prop:adjacencyspectrum} by the observation
$$\Delta_n = \Delta(\text{cd}^n\Delta_d) = (d+1)\cdot I - A(G_n) = (d+1)\cdot I - \Xi_n.$$
Thus the spectral cdfs satisfy the following condition:
$$\Lambda(\Delta_n)(x) = (d+1) - \Lambda(\Xi_n)(1-x).$$

So that by definition both limits are equal if for the shifting function $\sigma(x) := (d+1) - x$ we have
$$\mathcal{P}_i = \sigma(\mathcal{A}_i)$$
$$\mathcal{Q}_i = \sigma(\mathcal{B}_i)$$
for every $i$.

For $i = 0$ this is obviously true since $\sigma(0) = d+1$ and $\sigma(-2) = d+3$ and for $i > 0$ we can see this by the following inductive consideration:

Let $\lambda\in \mathcal{A}_i$ then $g^{-1}(\lambda)\subseteq\mathcal{A}_{i+1}$. Let $\lambda'\in\mathcal{A}_{i+1}$ be given so that
$$g(\lambda') = \lambda.$$
For $\mu = \sigma(\lambda)\in \mathcal{P}_i$ we obtain $\mu' = \sigma(\lambda')$ from $\mu$ over $f$ as follows:

\begin{eqnarray*}
d+1 - \mu = \lambda &=& g(\lambda')\\
&=& (d+1 - \mu')^2 - (d-1)(d+1-\mu') - (d+1)\\
&=& d+1 - (d+3)\mu' + (\mu')^2\\
&=& d+1 - f(\mu').
\end{eqnarray*}
Thus $\mu = f(\mu')$. The same calculation gives us the equality of $\mathcal{P}_i$ and $\sigma(\mathcal{A}_i)$. Thus the induction holds.

Analogously we can see $\mathcal{Q}_i = \sigma(\mathcal{B}_i)$ and so the limit of $\Lambda(\Delta_n)$ is the claimed step function.
\end{proof}

\newpage
\section{Relations to Fractal Theory}\label{sec:fractal}
Given a inclusion-uniform subdivision $\text{div}$ we associate to it the sequence $(\Gamma_n)_{n\in\mathbb{N}} = (\Gamma^{(d)}(\text{div}^n\Delta_d))_{n\in\mathbb{N}}$ of simple graphs.
This sequence can be considered self-similar when relaxing known constructions of graph-directed self-similar sets.
This relaxation has to be made to common definitions since fractals associated to inclusion-uniform subdivisions are not finitely-ramified in general and thus the orientation matters when joining graphs. In the following construction we will compensate for the ambiguity introduced by dependence on orientation. Note however that all considered examples can be oriented in a more convenient way - thus also giving rise to a Schreier graph approximation of the sequence.

The construction which dualizes iterated subdivision by a inclusion-uniform operation as a graph-sequence approximation of a self-similar set approximation is the following:\\

\textbf{The data:}
Let $d,n,N\in\mathbb{N}$.
We start with an initial graph $\Gamma_0$ on vertex set $\{1,...,N\}$ with degrees bounded by $d+1$ and a dedicated $S_{d+1}$-action on it. We will formally let $S_{d+1}$ act on $\{0,...,d\}$ for the purpose of this construction and denote action of $\sigma\in S_{d+1}$ on $i\in V(\Gamma_0) = \{1,...,N\}$ by $\sigma i$ opposed to the evaluation $\sigma(i)$ at $i\in \{0,...,d\}$. $\sigma_\ast: E(\Gamma_0)\rightarrow E(\Gamma_0)$ denotes the push forward action on the edge set, i.e.
$$\sigma_\ast \{v,w\} := \{\sigma v, \sigma w\}\in E(\Gamma_0)$$
for $\{v,w\}\in E(\Gamma_0)$.

Further let $\partial_i\Gamma_0$, $i=0,...,d$, denote $(d+1)$-many dedicated $n$-element boundary sets of $\Gamma_0$ such that $\partial\Gamma_0 := \bigcup_{i=0}^d\partial_i\Gamma_0$ is the set of vertices with degree $< d+1$ in $\Gamma_0$. Furthermore for every vertex $v\in \partial\Gamma_0$ we require that
$$\text{deg}v + \underbrace{\# \{i\in \{0,...,d\} ~|~ v\in\partial_i\Gamma_0\}}_{=:b_v} = d+1.$$
We further want the above sets $\partial_i\Gamma_0$ to be compatible with the group action in the sense that
\begin{equation}\tag{$\ast$}
\partial_i\Gamma_0 = \tau_{ij}\partial_j\Gamma_0
\end{equation}
for the transposition $\tau_{ij} = (i~j)$ and the set $\partial_i\Gamma_0$ has to be invariant under $S_d^i = \{\sigma\in S_{d+1} ~|~ \sigma(i) = i\}$, $i=0,...,d$.

In order to make $\Gamma_0$ $(d+1)$-regular we add $b_v$ many loops to the vertices $v\in \partial\Gamma_0$ and denote the loop added to $v\in\partial_i\Gamma_0$ for the $i$-th boundary by $\ell_v(i)$. The action of $S_{d+1}$ extends to this graph in a natural way by
$$\sigma_\ast \ell_v(i) := \ell_{\sigma v}(\sigma(i)).$$
We will thus denote the graph obtained by this addition as $\Gamma_0$ from now on.

For every loop $e = \ell_v(i)\in E(\Gamma_0)$ let $\kappa_v(e) := i$ and for every edge $e = \{v,w\}\in E(\Gamma_0)$ choose $\kappa_v(e), \kappa_w(e) \in \{0,...,d\}$ such that
$$\{\kappa_v(e) ~|~ v\in e \in E(\Gamma_0)\} = \{0,...,d\}.$$

Further let $\rho_{ij}: \{0,...,d\} \rightarrow \{0,...,d\}$ be a bijection such that
$$\rho_{ij}(\kappa_i(e)) = \kappa_j(e)$$
and $\rho_{ji} := \rho_{ij}^{-1}$ aswell as the compatibility with the $S_{d+1}$-action as $\rho_{\sigma i ~ \sigma j} = \nu_{\sigma, j} \rho_{ij} \nu_{\sigma, i}^{-1}$, where $\nu_\sigma\in S_{d+1}$ is given as
$$\nu_{\sigma, i}(\kappa_i(e)) := \kappa_{\sigma i}(\sigma_\ast e)$$

\textbf{The construction:} Having this information fixed we construct a self-similar sequence $(\Gamma_k)_{k\in\mathbb{N}}$.
Assume $\Gamma_{k-1}$ and $\partial_i\Gamma_{k-1}$ with a $S_{d+1}$-action sufficing the same conditions as the action on $\Gamma_0$ are constructed. Then we construct $\Gamma_k$ as a graph on vertex set $V(\Gamma_k) = V(\Gamma_0)\times V(\Gamma_{k-1}) = [N]^{k+1}$ and denote a vertex by a pair $(i,v)$ for $i\in V(\Gamma_0)$ and $v\in V(\Gamma_{k-1})$. 
We include the edge $\{(i,v), (j, w)\} \in E(\Gamma_k)$ if one of the following two conditions is met:
\begin{itemize}
    \item Either $i = j$ and $\{v,w\}\in E(\Gamma_{k-1})$ or
    \item $i\neq j$, $e = \{i,j\}\in E(\Gamma_0)$, $v\in\partial_{\kappa_i(e)}\Gamma_{k-1}$, $w\in\partial_{\kappa_j(e)}\Gamma_{k-1}$ and
    $$\rho_{ij} v = w.$$
\end{itemize}
Furthermore we include the loops at vertices $(i,v)$ for every loop at $v$ in $\Gamma_{k-1}$.

We will now show that $\Gamma_k$ again admits an $S_{d+1}$-action and construct sets $\partial_i\Gamma_k$. The sets $\partial_i\Gamma_k$ are given by
$$\partial_i\Gamma_k := \partial_i\Gamma_{k-1}\times \partial_i\Gamma_0 = (\partial_i\Gamma_0)^{k+1}.$$
The graph $\Gamma_k$ can be assigned an $S_{d+1}$-action as follows:
$$\sigma\cdot (i,v) := (\sigma i, \nu_{\sigma,i} v),$$
where $\nu_{\sigma, i}$ as above is the permutation in $S_{d+1}$ given by $\nu_{\sigma,i}: \kappa_i(e) \mapsto \kappa_{\sigma i}(\sigma_\ast e)$ for all $e\in E(\Gamma_0)$ with $i\in e$.
We will now show that this action is compatible with the multiplication of $S_{d+1}$. To this end note that $\nu_{\text{id},i}(\kappa_i(e)) = \kappa_i(e)$ and thus $\nu_{\text{id}, i} = \text{id}$. Let $\sigma,\tau\in S_{d+1}$ then
$$\nu_{\sigma\tau, i}(\kappa_i(e)) = \kappa_{(\sigma\tau) i}((\sigma\tau)_\ast e) = \kappa_{\sigma(\tau i)}(\sigma_\ast (\tau_\ast e)) = \nu_{\sigma, \tau i}(\kappa_{\tau i}(\tau_\ast e)) = \nu_{\sigma,\tau i}(\nu_{\tau, i}(\kappa_i(e)))$$
such that $\nu_{\sigma\tau, i} = \nu_{\sigma, \tau i}\nu_{\tau, i}$.

In particular we have
$$(\sigma\tau)(i, v) = ((\sigma\tau) i, \nu_{\sigma\tau, i} v) = (\sigma (\tau i), \nu_{\sigma, \tau i} (\nu_{\tau, i} v)) = \sigma (\tau i, \nu_{\tau, i} v) = \sigma (\tau (i, v)).$$
Thus we obtain a well-defined action of $S_{d+1}$ on $V(\Gamma_k)$. We will show that it preserves edge relations, which is immediate for edges of the first type.

Assume we have an edge of the second type, i.e. $e = \{(i,v), (j, w)\}$ and $i \neq j$, $e' = \{i, j\}\in E(\Gamma_0)$, $v\in\partial_{\kappa_i(e')}\Gamma_{k-1}$, $w\in\partial_{\kappa_j(e')}\Gamma_{k-1}$ and
$$\rho_{ij} v = w.$$
Let $\sigma\in S_{d+1}$ be given. Obviously $\sigma i\neq \sigma j$ and $\sigma_\ast e' \in E(\Gamma_0)$. Since for $\ell,\ell'\in \{0,...,d\}$ $\partial_\ell\Gamma_{k-1}$ are invariant under $S_d^\ell$ and $\partial_\ell\Gamma_{k-1} = \tau_{\ell\ell'}\partial_\ell'\Gamma_{k-1}$ we have
$$\sigma \partial_\ell\Gamma_{k-1} = ((\ell ~ \sigma(\ell))\circ \tau) \partial_\ell\Gamma_{k-1} = \partial_{\sigma(\ell)} \Gamma_{k-1}$$
for $\tau\in S_d^\ell$ given as $\tau(p) = \sigma(p)$ if $p\notin\{\ell,\sigma^{-1}(\ell)\}$ and $\tau(\ell) = \ell$, $\tau(\sigma^{-1}(\ell)) = \sigma(\ell)$.
Denote now $\ell = \kappa_i(e')$, $\ell' = \kappa_j(e')$ so that $\nu_{\sigma, i} v \in \partial_{\nu_{\sigma, i}(\ell)} \Gamma_{k-1}$ and $\nu_{\sigma, j} w\in \partial_{\nu_{\sigma, j}(\ell')}\Gamma_{k-1}$. 
Thus the last condition to show is that
$$\rho_{\sigma i~ \sigma j}(\nu_{\sigma, i}(v)) = \nu_{\sigma, j}(w).$$
This is by definition equivalent to
$$(\nu_{\sigma, j}^{-1} \rho_{\sigma i~\sigma j} \nu_{\sigma, i})v = w = \rho_{ij} v.$$
But by the condition on $\rho_{ij}$ we have
$$\rho_{\sigma i~\sigma j} = \nu_{\sigma, j}\rho_{ij}\nu_{\sigma, i}^{-1}$$
showing the claim.

So that all conditions for $\Gamma_k$ are met in order to iteratively define the next graph in the sequence.

\textbf{Duality to iterated subdivision:}
Let $\text{div}$ be a inclusion-uniform subdivision acting on $d$-dimensional complexes. 
We apply the above construction with the given $d$, $n = f_{d-1}(\text{div})$ and $N = f_d(\text{div})$. We can equip the $d$-dual graph of the subdivision of $\Delta_d$,
$$\Gamma_0 := \Gamma^{(d)}(\text{div}\Delta_d),$$
with an $S_{d+1}$-action by the condition of being inclusion-uniform; let $\sigma\in S_{d+1}$ and $K = \Delta_d$ in Definition \ref{def:barycentric}; then $\sigma$ canonically defines a bijective vertex-identification
$$\sigma: F_0(\Delta_d) \rightarrow F_0(\Delta_d)$$
and thus extends to a (geometric) simplicial isomorphism of
$$\tilde{\sigma}: \text{div}\Delta_d \rightarrow \text{div}\Delta_d$$
sufficing $\tilde{\sigma}(\{i\}) = \{\sigma(i)\}$. Note that $\tilde{\sigma}$ defines an action on $\Gamma_0$ by
$$\sigma \tau := \tilde{\sigma}(\tau)$$
for $\tau\in V(\Gamma_0) = F_d(\text{div}\Delta_d)$.
Obviously this action preserves the edge-relation since $\tilde{\sigma}$ is a simplicial isomorphism.

We enumerate the facets $F_d(\text{div}\Delta_d) = \{\tau_1,...,\tau_N\}$ and assume $\tau_i$ corresponds to the vertex $i\in V(\Gamma_0)$.

In order to define the boundary sets let $\sigma_i = (0,...,\hat{i},...,d)\in F_{d-1}(\Delta_d)$ and let
$$\Sigma_i := F_{d-1}(\text{div}_{\Delta_d}\sigma_i),$$
where $s: \text{div}\Delta_d\rightarrow \Delta_d$ denotes the subdivision map. Then let
$$\partial_i\Gamma_0 := \{\tau\in F_d(\text{div}\Delta_d) ~|~ F_{d-1}(\tau)\cap \Sigma_i \neq \emptyset\}.$$

Obviously since the isomorphism is geometric it has to restrict to an isomorphism on the boundaries; thus $\tilde{\tau}_{ij}$ is mapping $\Sigma_i$ to $\Sigma_j$ and so it maps the respective unique facets of the $(d-1)$-faces in $\Sigma_i$ to their counterparts in $\Sigma_j$, i.e. 
$$\partial_i\Gamma_0 = \tau_{ij}\partial_j\Gamma_0.$$

Furthermore assuming $\sigma\in S_d^i$ we have that it leaves $\sigma_i$ invariant as a set of vertices so that $\sigma\Sigma_i = \Sigma_i$ and analogously
$$\sigma \partial_i\Gamma_0 = \partial_i\Gamma_0.$$

The labels $\kappa_i(e)$, $i\in [N]$, can be chosen at will respecting the condition that $\kappa_i(\ell_i(j)) = j$ if $i\in\partial_j\Gamma_0$.

We now identify an edge $\{i,j\}\in E(\Gamma_0)$ with the face $\tau_i\cap\tau_j\in F_{d-1}(\text{div}\Delta_d)$ for their respective facets $\tau_i,\tau_j\in F_d(\text{div}\Delta_d)$.
The bijections $\rho_{ij}$, $\{i,j\}\in E(\Gamma_0)$, are then given by the maps
$$\rho_{ij}(\kappa_i(\nu)) := \kappa_j(\nu')$$
for $\nu\in F_{d-1}(\tau_i)$ fixed and $\nu'$ the unique face in $F_{d-1}(\tau_j)$ such that $\nu\cap\nu'\in F_{d-2}(\tau_i\cap\tau_j)$.

The equation for $\rho_{ij}$ making it compatible with the $S_{d+1}$-action can be seen as follows; assume $\{i,j\}\in E(\Gamma_0)$ and let $\nu\in F_{d-1}(\tau_i)$ be given and $\nu'\in F_{d-1}(\tau_j)$ be the unique face such that $\nu\cap\nu'\in F_{d-2}(\tau_i\cap \tau_j)$. By definition of the action $\sigma\in S_{d+1}$ acts on $\Gamma_0$ as the isomorphism $\tilde{\sigma}$ acts on the facets. Since $\tilde{\sigma}$ is a simplicial isomorphism we have that
$$\tilde{\sigma}(\nu)\cap\tilde{\sigma}(\nu') \in F_{d-2}(\tilde{\sigma}(\tau_i)\cap\tilde{\sigma}(\tau_j)).$$
Furthermore we know that if the edges $e,e'\in E(\Gamma_0)$ correspond to $\nu, \nu'$, respectively, then the edges corresponding to $\tilde{\sigma}(\nu)$ and $\tilde{\sigma}(\nu')$ are given by $\sigma_\ast e$ and $\sigma_\ast e'$ by definition, respectively.

Thus
$$\rho_{\sigma i~\sigma j}(\nu_{\sigma, i}(\kappa_i(e))) = \rho_{\sigma i ~\sigma j}(\kappa_{\sigma i}(\sigma_\ast e)) = \kappa_{\sigma j}(\sigma_\ast e') = \nu_{\sigma, j}(\kappa_j(e')) = \nu_{\sigma, j}(\rho_{ij}(\kappa_i(e))),$$
showing the compatibility with the $S_{d+1}$-action.

\begin{theorem}
In the above setting it holds that $\tilde{\Gamma}_k \cong \Gamma^{(d)}(\text{\normalfont div}^{k+1}\Delta_d)$ for all $k \geq 0$, where $\tilde{\Gamma}_k$ results from $\Gamma_k$ by removing loops.

In case $\text{\normalfont div}$ acts non-trivial on $d$-faces the graph $\Gamma_k$ contains $(d+1) f_{d-1}(\text{\normalfont div})^{k+1}$ loops. In particular $\Gamma_k$ approximates $\Gamma^{(d)}(\text{\normalfont div}^{k+1}\Delta_d)$.
\end{theorem}
\begin{proof}
The claim is trivially true for $k = 0$. Thus let $k > 0$ and assume the claim is satisfied for $k-1$.

Let $i\in [N]$ be fixed for now. Note that the choice of $\kappa_i(e)$ for $e\in E(\Gamma_0)$, $i\in e$, corresponds to an ordering of the vertices of $\tau_i$ as follows:

For $e\in E(\Gamma_0)$ with $i\in e$ let $\nu_e\in F_{d-1}(\tau_i)$ denote the face generating $e$ in $\Gamma_0$, i.e. if $e = \{i,j\}$ for $j\in [N]$ then $\nu_e = \tau_i\cap \tau_j$ and if $e = \ell_i(j)$ for some $j\in \{0,...,d\}$ we let $\nu_e$ denote the unique face in $\Sigma_j\cap F_{d-1}(\tau_i)$.
Now we order the vertices of $\tau_i$ in a way compatible with how we ordered the boundary sets - i.e. we let $v\in F_0(\tau_i)$ be at position
$$\kappa_i(e)$$
where $e\in E(\Gamma_0)$, $i\in e$, is the unique edge such that $\nu_e = \tau_i\setminus \{v\}$. This vertex will be denoted $v_{\kappa_i(e)}^i$ from now on such that
$$\tau_i = (v_0^i,...,v_d^i)$$
is now an ordered simplex.

Now note that
$$\text{div}^{k+1}\Delta_d = \text{div}^k \text{div} \Delta_d.$$
Furthermore note that the subdivision procedure $\text{div}^k$ is itself inclusion-uniform and thus for every face $\tau_i = \{v_0,...,v_d\} \in \text{div}\Delta_d$ and bijection
$$f: \{v_0,...,v_d\} \rightarrow \{0,...,d\}$$
there exists a unique isomorphism
$$f^{(k)}: \text{div}^k_{\text{div}\Delta_d}\tau_i \rightarrow \text{div}^k\Delta_d,$$
such that $f^{(k)}(v_j) = f(v_j)$, $j = 0,...,d$.
In particular the dual graph of $\text{div}^k_{\text{div}\Delta_d}\tau_i$ is isomorphic to $\Gamma_{k-1}$ by induction hypothesis.
Note that the dual graph of $\text{div}^k_{\text{div}\Delta_d}\tau_i$ is the restriction of $\Gamma^{(d)}(\text{div}^{k+1}\Delta_d)$ to the set of facets added in the interior of $\tau_i$.

Thus we take $N$ copies of $\Gamma_{k-1}$ - one for each facet $\tau_i$. In order to obtain an isomorphism $\tilde{\Gamma}_k \cong \Gamma^{(d)}(\text{div}^{k+1}\Delta_d)$ we only need to show that the edges of second kind in the construction are indeed the edges obtained by gluing the copies of $\text{div}^k\Delta_d$ along their boundaries.

To this end assume two faces $\tau_i, \tau_j\in F_d(\text{div}\Delta_d)$ are given such that $e = \{i,j\}\in E(\Gamma_0)$. By definition they meet in the common face
$$\tau_i \setminus \{v_{\kappa_i(e)}^i\} = \tau_j\setminus\{v_{\kappa_j(e)}^j\}.$$

Note that the action of $\rho_{ij}$ induced on $\tau_i$ maps $v_{\kappa_i(e)}^i$ to $v_{\kappa_j(e)}^j$. $\rho_{ij}$ delivers even more; by definition the action of $\rho_{ij}$ on the vertices maps $v_{\kappa_i(e')}^i$ to $v_{\kappa_j(e'')}^j$ whenever the vertices opposed to $e'$ and $e''$ in $i$ and $j$, respectively, are geometrically identical (i.e. when they are identified in the gluing process). This can be seen by the $(d-2)$-adjacency of the edges $e'$ and $e''$ (or rather their generating faces) in the boundary of $\tau_i\cap\tau_j$.
In particular if we consider $\text{div}\Delta_d$ to be obtained by a gluing $\mathcal{G}_\ast(\sigma_1,...,\sigma_N)$ for $N$ copies of the standard simplex $\Delta_d$ with $\sigma_i$ corresponding to $\tau_i$ in the glued complex $\mathcal{G}_\ast(\sigma_1,...,\sigma_N) \cong \text{div}\Delta_d$. We will identify $\sigma_i$ with $\tau_i$ by the ordering fixed above; i.e. the canonical inclusion of the $i$-th standard simplex is given by
$$\iota_i: \ell \mapsto v^i_{\ell}.$$
Under this identification the restriction of $\rho_{ij}$ to $\{\kappa_i(e') ~|~ e'\in E(\Gamma_0), ~i\in e'\} \setminus \{\kappa_i(e)\}$ is thus precisely the map
$$\iota_j^{-1}\circ\iota_i$$
and so by definition of the $S_{d+1}$-action acts as its extended isomorphism
$$\widetilde{\iota_j^{-1}\circ \iota_i}$$
which by equation (\ref{eq:gluing_bij}) from Section \ref{sec:prel} gives exactly the vertex bijection of the relation $\mathcal{G}'$ for obtaining the subdivision as the glued complex
$$\mathcal{G}'_\ast(\text{div}^k\sigma_1,...,\text{div}^k\sigma_N)$$
which is isomorphic to the graph $\tilde{\Gamma}_k$ by how $\rho_{ij}$ identifies the boundaries.
\end{proof}

Note that this fractal process can be applied not only to $\Gamma_0$ being the dual graph of $\text{div}\Delta_d$. Having constructed this sequence of fractals for $\Gamma_0 = \Gamma^{(d)}(\text{div}\Delta_d)$. We can also define fractal sequences on any given pseudo-manifold $K$ in the same fashion by fixing the same maps $\kappa_i$ and $\rho_{ij}$ (which amounts to chosing an ordering for every simplex in $K$) and perform gluing along the boundary by utilizing the $S_{d+1}$-action on the already constructed sequence of subdivided standard simplices. The fractal sequence arising from this is the sequence of dual graphs of the iterated subdivisions of $K$.

In Figure \ref{fig:fractals} we have illustrated the input data for $\Gamma_0$, $\kappa$ and $\rho$ in order to generate the barycentric or edgewise subdivision (with parameter $3$) of a $2$-simplex respectively.

\begin{figure}[h]
    \begin{subfigure}[t]{.5\textwidth}
          \centering
          \tikzset{every picture/.style={line width=0.75pt}} 

\begin{tikzpicture}[x=0.75pt,y=0.75pt,yscale=-1,xscale=1]

\draw   (320,40) -- (460,260) -- (180,260) -- cycle ;
\draw    (320,40) -- (320,260) ;
\draw    (250,150) -- (460,260) ;
\draw    (390,150) -- (180,260) ;
\draw    (290,230) -- (260,180) ;
\draw    (290,230) -- (350,230) ;
\draw    (380,180) -- (350,230) ;
\draw    (290,130) -- (260,180) ;
\draw    (350,130) -- (380,180) ;
\draw    (290,130) -- (350,130) ;
\draw   (249,101) .. controls (258,89) and (273,111) .. (290,130) .. controls (267,121) and (240,113) .. (249,101) -- cycle ;
\draw   (217,153) .. controls (226,141) and (243,161) .. (260,180) .. controls (235,173) and (208,165) .. (217,153) -- cycle ;
\draw   (392,101) .. controls (398,113) and (377,120) .. (350,130) .. controls (369,112) and (385,91) .. (392,101) -- cycle ;
\draw   (423,153) .. controls (429,165) and (404,171) .. (380,180) .. controls (399,162) and (416,143) .. (423,153) -- cycle ;
\draw   (350,282) .. controls (338,281) and (344,246) .. (350,230) .. controls (356,245) and (365,282) .. (350,282) -- cycle ;
\draw   (290,282) .. controls (278,281) and (284,246) .. (290,230) .. controls (296,245) and (305,282) .. (290,282) -- cycle ;
\draw  [dash pattern={on 4.5pt off 4.5pt}]  (240,180) -- (279.01,248.26) ;
\draw [shift={(280,250)}, rotate = 240.26] [color={rgb, 255:red, 0; green, 0; blue, 0 }  ][line width=0.75]    (10.93,-3.29) .. controls (6.95,-1.4) and (3.31,-0.3) .. (0,0) .. controls (3.31,0.3) and (6.95,1.4) .. (10.93,3.29)   ;
\draw  [dash pattern={on 4.5pt off 4.5pt}]  (273,173) -- (299,218.27) ;
\draw [shift={(300,220)}, rotate = 240.12] [color={rgb, 255:red, 0; green, 0; blue, 0 }  ][line width=0.75]    (10.93,-3.29) .. controls (6.95,-1.4) and (3.31,-0.3) .. (0,0) .. controls (3.31,0.3) and (6.95,1.4) .. (10.93,3.29)   ;

\draw (315,22.4) node [anchor=north west][inner sep=0.75pt]    {$0$};
\draw (167,261.4) node [anchor=north west][inner sep=0.75pt]    {$1$};
\draw (462,261.4) node [anchor=north west][inner sep=0.75pt]    {$2$};
\draw (304,300.4) node [anchor=north west][inner sep=0.75pt]    {$\partial _{0} \Gamma _{0}$};
\draw (181,102.4) node [anchor=north west][inner sep=0.75pt]    {$\partial _{2} \Gamma _{0}$};
\draw (431,102.4) node [anchor=north west][inner sep=0.75pt]    {$\partial _{1} \Gamma _{0}$};
\draw (347,288.4) node [anchor=north west][inner sep=0.75pt]  [font=\scriptsize]  {$0$};
\draw (287,288.4) node [anchor=north west][inner sep=0.75pt]  [font=\scriptsize]  {$0$};
\draw (397,92.4) node [anchor=north west][inner sep=0.75pt]  [font=\scriptsize]  {$1$};
\draw (427,144.4) node [anchor=north west][inner sep=0.75pt]  [font=\scriptsize]  {$1$};
\draw (238,92.4) node [anchor=north west][inner sep=0.75pt]  [font=\scriptsize]  {$2$};
\draw (207,142.4) node [anchor=north west][inner sep=0.75pt]  [font=\scriptsize]  {$2$};
\draw (331,118.4) node [anchor=north west][inner sep=0.75pt]  [font=\scriptsize]  {$0$};
\draw (256,161.4) node [anchor=north west][inner sep=0.75pt]  [font=\scriptsize]  {$0$};
\draw (301,118.4) node [anchor=north west][inner sep=0.75pt]  [font=\scriptsize]  {$0$};
\draw (268,138.4) node [anchor=north west][inner sep=0.75pt]  [font=\scriptsize]  {$1$};
\draw (363,137.4) node [anchor=north west][inner sep=0.75pt]  [font=\scriptsize]  {$2$};
\draw (377,161.4) node [anchor=north west][inner sep=0.75pt]  [font=\scriptsize]  {$0$};
\draw (375,195.4) node [anchor=north west][inner sep=0.75pt]  [font=\scriptsize]  {$2$};
\draw (362,217.4) node [anchor=north west][inner sep=0.75pt]  [font=\scriptsize]  {$2$};
\draw (330,232.4) node [anchor=north west][inner sep=0.75pt]  [font=\scriptsize]  {$1$};
\draw (303,232.4) node [anchor=north west][inner sep=0.75pt]  [font=\scriptsize]  {$2$};
\draw (272,218.4) node [anchor=north west][inner sep=0.75pt]  [font=\scriptsize]  {$1$};
\draw (257,195.4) node [anchor=north west][inner sep=0.75pt]  [font=\scriptsize]  {$1$};

\end{tikzpicture}
          \caption{Barycentric subdivision $\text{div} = \text{sd}$ and $d = 2$.}
    \end{subfigure}
    \hfill
    \begin{subfigure}[t]{.5\textwidth}
          \centering
          \tikzset{every picture/.style={line width=0.75pt}} 

\begin{tikzpicture}[x=0.75pt,y=0.75pt,yscale=-1,xscale=1]

\draw   (340,30) -- (480,247) -- (200,247) -- cycle ;
\draw   (299,48) .. controls (308,36) and (323,58) .. (340,77) .. controls (317,68) and (290,60) .. (299,48) -- cycle ;
\draw   (251,120) .. controls (260,108) and (277,128) .. (294,147) .. controls (269,140) and (242,132) .. (251,120) -- cycle ;
\draw   (382,48) .. controls (388,60) and (367,67) .. (340,77) .. controls (359,59) and (375,38) .. (382,48) -- cycle ;
\draw   (430,120) .. controls (436,132) and (411,138) .. (387,147) .. controls (406,129) and (423,110) .. (430,120) -- cycle ;
\draw   (250,269) .. controls (238,268) and (244,233) .. (250,217) .. controls (256,232) and (265,269) .. (250,269) -- cycle ;
\draw   (436,269) .. controls (424,268) and (430,233) .. (436,217) .. controls (442,232) and (451,269) .. (436,269) -- cycle ;
\draw    (248,173) -- (290,247) ;
\draw    (248,173) -- (433,173) ;
\draw    (290,107) -- (390,107) ;
\draw    (290,107) -- (340.5,173) ;
\draw    (340.5,173) -- (390,107) ;
\draw    (340.5,173) -- (290,247) ;
\draw    (340.5,173) -- (390,247) ;
\draw    (433,173) -- (390,247) ;
\draw   (479,190) .. controls (485,202) and (460,208) .. (436,217) .. controls (455,199) and (472,180) .. (479,190) -- cycle ;
\draw   (207,190) .. controls (216,178) and (233,198) .. (250,217) .. controls (225,210) and (198,202) .. (207,190) -- cycle ;
\draw   (340,269) .. controls (328,268) and (334,233) .. (340,217) .. controls (346,232) and (355,269) .. (340,269) -- cycle ;
\draw    (250,217) -- (294,197) ;
\draw    (387,197) -- (436,217) ;
\draw    (294,197) -- (340,217) ;
\draw    (340,217) -- (387,197) ;
\draw    (387,147) -- (387,197) ;
\draw    (294,147) -- (294,197) ;
\draw    (294,147) -- (340,127) ;
\draw    (387,147) -- (340,127) ;
\draw    (340,77) -- (340,127) ;
\draw  [dash pattern={on 4.5pt off 4.5pt}]  (282,212) -- (328.13,229.3) ;
\draw [shift={(330,230)}, rotate = 200.56] [color={rgb, 255:red, 0; green, 0; blue, 0 }  ][line width=0.75]    (10.93,-3.29) .. controls (6.95,-1.4) and (3.31,-0.3) .. (0,0) .. controls (3.31,0.3) and (6.95,1.4) .. (10.93,3.29)   ;
\draw  [dash pattern={on 4.5pt off 4.5pt}]  (302,182) -- (348.13,199.3) ;
\draw [shift={(350,200)}, rotate = 200.56] [color={rgb, 255:red, 0; green, 0; blue, 0 }  ][line width=0.75]    (10.93,-3.29) .. controls (6.95,-1.4) and (3.31,-0.3) .. (0,0) .. controls (3.31,0.3) and (6.95,1.4) .. (10.93,3.29)   ;

\draw (335,12.4) node [anchor=north west][inner sep=0.75pt]    {$0$};
\draw (187,251.4) node [anchor=north west][inner sep=0.75pt]    {$1$};
\draw (482,251.4) node [anchor=north west][inner sep=0.75pt]    {$2$};
\draw (247,275.4) node [anchor=north west][inner sep=0.75pt]  [font=\scriptsize]  {$0$};
\draw (337,275.4) node [anchor=north west][inner sep=0.75pt]  [font=\scriptsize]  {$0$};
\draw (433,275.4) node [anchor=north west][inner sep=0.75pt]  [font=\scriptsize]  {$0$};
\draw (387,39.4) node [anchor=north west][inner sep=0.75pt]  [font=\scriptsize]  {$1$};
\draw (435,111.4) node [anchor=north west][inner sep=0.75pt]  [font=\scriptsize]  {$1$};
\draw (484,180.4) node [anchor=north west][inner sep=0.75pt]  [font=\scriptsize]  {$1$};
\draw (197,179.4) node [anchor=north west][inner sep=0.75pt]  [font=\scriptsize]  {$2$};
\draw (241,109.4) node [anchor=north west][inner sep=0.75pt]  [font=\scriptsize]  {$2$};
\draw (287,39.4) node [anchor=north west][inner sep=0.75pt]  [font=\scriptsize]  {$2$};
\draw (251,202.4) node [anchor=north west][inner sep=0.75pt]  [font=\scriptsize]  {$1$};
\draw (342,80.4) node [anchor=north west][inner sep=0.75pt]  [font=\scriptsize]  {$0$};
\draw (427,202.4) node [anchor=north west][inner sep=0.75pt]  [font=\scriptsize]  {$2$};
\draw (390,151.4) node [anchor=north west][inner sep=0.75pt]  [font=\scriptsize]  {$0$};
\draw (377,132.4) node [anchor=north west][inner sep=0.75pt]  [font=\scriptsize]  {$2$};
\draw (344,110.4) node [anchor=north west][inner sep=0.75pt]  [font=\scriptsize]  {$0$};
\draw (317,122.4) node [anchor=north west][inner sep=0.75pt]  [font=\scriptsize]  {$1$};
\draw (359,124.4) node [anchor=north west][inner sep=0.75pt]  [font=\scriptsize]  {$2$};
\draw (391,181.4) node [anchor=north west][inner sep=0.75pt]  [font=\scriptsize]  {$0$};
\draw (372,205.4) node [anchor=north west][inner sep=0.75pt]  [font=\scriptsize]  {$1$};
\draw (406,194.4) node [anchor=north west][inner sep=0.75pt]  [font=\scriptsize]  {$2$};
\draw (354,212.4) node [anchor=north west][inner sep=0.75pt]  [font=\scriptsize]  {$1$};
\draw (321,212.4) node [anchor=north west][inner sep=0.75pt]  [font=\scriptsize]  {$2$};
\draw (298,205.4) node [anchor=north west][inner sep=0.75pt]  [font=\scriptsize]  {$2$};
\draw (284,175.4) node [anchor=north west][inner sep=0.75pt]  [font=\scriptsize]  {$0$};
\draw (271,193.4) node [anchor=north west][inner sep=0.75pt]  [font=\scriptsize]  {$1$};
\draw (284,160.4) node [anchor=north west][inner sep=0.75pt]  [font=\scriptsize]  {$0$};
\draw (299,130.4) node [anchor=north west][inner sep=0.75pt]  [font=\scriptsize]  {$1$};
\draw (324,290.4) node [anchor=north west][inner sep=0.75pt]    {$\partial _{0} \Gamma _{0}$};
\draw (201,92.4) node [anchor=north west][inner sep=0.75pt]    {$\partial _{2} \Gamma _{0}$};
\draw (451,92.4) node [anchor=north west][inner sep=0.75pt]    {$\partial _{1} \Gamma _{0}$};

\end{tikzpicture}
          \caption{Edgewise subdivision for parameter $r = 3$ and $d = 2$.}
    \end{subfigure}
    
    \caption{Two subdivision procedure of infinite ramification and their respective $\Gamma_0$ with added loops. The dashed arrow indicates the map $\rho_{ij}$ for a particular edge $\{i,j\}$. The numbers along the edge indicate the values of $\kappa_i$ associated to the vertex of the edge closer to the label.}\label{fig:fractals}
\end{figure}
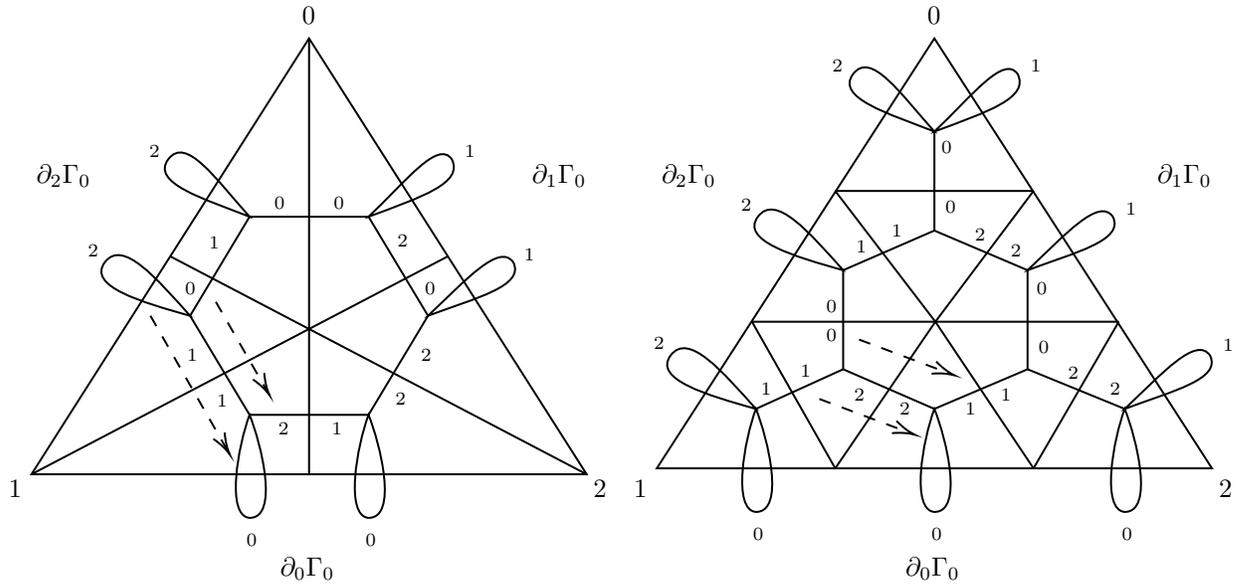

\textbf{In case of finite ramification:}
The generic case of finite ramification is the case where $n = 1$. This is due to the fact that we call a self-similar set construction of the above type \textit{finitely ramified} if every copy of $\Gamma_{k-1}$ in $\Gamma_k$ can be isolated by the removal of a bounded number of edges (independent of $k$). However the boundaries to be joined have $n^{k+1}$ elements which is only bounded by a constant if $n = 1$.

In this case the above construction reduces to a construction related to a graph sequence approximating a self-similar set in the sense of Sabot, \cite{sabot}. We assume $\Gamma_0$ to be equipped with the enumeration $\kappa_i$ of edges at every vertex $i\in V(\Gamma_0) = [N]$ and view $\Gamma_0$ as generated by a relation $\mathcal{R}$ on the set $[d+1]\times [N]$, i.e. $\mathcal{R}$ is generated by the set of relations
$$(\kappa_i(\{i,j\}), i) \mathcal{R} (\kappa_j(\{i,j\}), j)$$
for every $\{i,j\}\in E(\Gamma_0)$.

Since $n = 1$ we can identify $\partial_i\Gamma_{k-1}$ with $\partial_i\Gamma_0$ and thus push the equivalence relation $\mathcal{R}$ from $\Gamma_0$ to the $k$-th level. Let $\Gamma_{k-1}^i$ denote the $i$-th copy of $\Gamma_{k-1}$ in $\Gamma_k$. We then join the vertex in $\partial_r \Gamma_{k-1}^i$ with the vertex in $\partial_\ell\Gamma_{k-1}^j$ iff
$$(r,i)\mathcal{R} (\ell, j),$$
i.e. iff $r = \kappa_i(\{i,j\})$, $\ell = \kappa_j(\{i,j\})$ and $\{i,j\}\in E(\Gamma_0)$.
Note that $\rho_{ij}$ does not play a role here since the restriction of the isomorphism induced by $\rho_{ij}$ over the $S_{d+1}$-action on $\partial_{\kappa_i(\{i,j\})}\Gamma_{k-1}$ then just maps this singleton onto the singleton $\partial_{\kappa_j(\{i,j\})}\Gamma_{k-1}$.

The above construction can then be transferred to the setting of Sabot by taking the line graph and adjusting the elementary cell $\Gamma_0$ accordingly. By \cite{spectral_effects} the spectral effects of taking the line graph is known in case the graph is regular.

\bibliographystyle{abbrvurl}
\bibliography{references}

\end{document}